\documentclass[10pt,a4paper]{amsart}
\usepackage[cp1250]{inputenc}
\usepackage[english]{babel}
\usepackage{color}
\usepackage[colorlinks,linkcolor=blue,citecolor=red]{hyperref}
\usepackage{amssymb}
\usepackage{amsmath}
\usepackage{lmodern}
\usepackage[IL2]{fontenc}
\usepackage{url}
\usepackage{graphicx}

\usepackage{pgf,tikz}
\usepackage{mathrsfs}
\usetikzlibrary{arrows}
\usetikzlibrary{snakes}
\usetikzlibrary{arrows}
\usetikzlibrary{calc}
\usetikzlibrary{patterns}
\usepackage{tkz-fct}

\newtheorem{defi}{\bf D\scriptsize EFINITION \normalsize}
\newtheorem{theorem}{\bf T\scriptsize HEOREM \normalsize}
\newtheorem{lm}{\bf L\scriptsize EMMA \normalsize}
\newtheorem{dk}{\bf C\scriptsize OROLLARY \normalsize}
\newtheorem{conj}{\bf C\scriptsize ONJECTURE \normalsize}
\newtheorem{rem}{\bf R\scriptsize EMARK \normalsize}
\newtheorem{exa}{\bf E\scriptsize XAMPLE \normalsize}
\newtheorem{pro}{\bf P\scriptsize ROBLEM \normalsize}
\newtheorem{prop}{\bf P\scriptsize ROPOSITION \normalsize}
\newtheorem{no}{\bf N\scriptsize OTE \normalsize}

\newenvironment{remark}{\begin{rem}\rm}{\end{rem}}
\def\kopr{\hfill\raisebox{3pt}{\framebox{$\star$}}}
\newenvironment{example}{\begin{exa}\rm}{$\kopr$\end{exa}}
\newenvironment{definition}{\begin{defi}\rm}{\end{defi}}

\newenvironment{corollary}{\begin{dk}\it}{\end{dk}}
\newenvironment{conjecture}{\begin{conj}\it}{\end{conj}}
\newenvironment{proposition}{\begin{prop}\it}{\end{prop}}

\newcommand{\nadsebou}[2]{\begin{array}{c} #1 \\ #2 \end{array}}

\newcommand{\zav}[1]{\left( #1 \right)}
\newcommand{\szav}[1]{\left\{ #1 \right\}}

\newcommand{\imag}{{\rm i}}
\newcommand{\serovna}[1]{\stackrel{(\ref{#1})}{=}}

\newcommand{\abs}[1]{\left| #1 \right|}

\newcommand{\hypop}{\operatorname*{\mathcal{H}}}
\renewcommand\Re{\operatorname{Re}}

\newcommand{\N}{\mathbb{N}}

\newcommand{\C}{{\mathbb{C}}}

\newcommand{\dd}{{\rm d}}
\newcommand{\ii}{{\rm i}}
\allowdisplaybreaks
\numberwithin{equation}{section}
\begin{document}
\title{Towards a change of variable formula for ``hypergeometrization''}
 \author{Petr Blaschke}
\address{ Mathematical Institute, Silesian University in Opava, Na Rybnicku 1, 746 01 Opava, Czech Republic}

\email{Petr.Blaschke@math.slu.cz}
\begin{abstract} 
We are going to study properties of ``hypergeometrization'' -- an operator which act on analytic functions near the origin by inserting two Pochhammer symbols into their Taylor series. In essence, this operator maps elementary function into hypergeometric. The main goal is to produce number of ``change of variable'' formulas for this operator which, in turn, can be used to derive great number of transform for multivariate hypergeometric functions.
\end{abstract}
\maketitle
\section{Introduction}
Hypergeometric functions and their multivariate analogs are well studied objects in mathematics. The classical references include \'Erdelyi \cite{Erdelyi}, Luke \cite{Luke}, Bailey \cite{Bailey}, Slater \cite{Slater} just to mention few. A very nice survey article about multivariate hypergeometric function of ``Appell's type'' was written by M. Schlosser in \cite{Schlosser}.

There are numerous ways how to extend hypergeometric function into higher dimension. There are Appell's function \cite{Appell}. Functions from the Horn's list \cite{Horn}, Kampé de Feriét functions \cite{Kampe,Appell2}, Lauricella functions \cite{Lauricella}, Srivastava function \cite{Karlsson}, Saran's functions \cite{Saran,Saran2}, 
$A$-hypergeomtric function \cite{Gelfand, Beukers,Beukers2}, hypergeometric functions of matrix argument \cite{macdonald, Blaschke5}, and so on.

These functions appears surprisingly often in all of analysis and have many application, e.g. in quantum field theory, in computing of Feynman integrals (see e.g. \cite{Shpot}), even appear also in chemistry \cite{dlmf8}. Recently a Karlsson's $FD_1$ function  \cite{Karlsson,Exton1,Exton2} appeared in the literature \cite{englis} in the context of harmonic Bergman spaces. 
   
The main object of study for these functions are various ``transforms'' i.e. identities that relates two of them together or one function to itself but with different values of parameters and/or argument(s).

A common feature of all of the mentioned functions (safe for functions of matrix argument) is the presence of a Pochhammer symbol, i.e. the quantity $(a)_k:=a(a+1)\cdots (a+k-1)$ in their series expansion.

It is therefore only natural to study a linear operator $\hypop_c^a$ called ``hypergeometrization'' depending on two complex parameters $a,c\in\mathbb{C}$ which acts on analytic functions near the origin by inserting two Pochhammer symbols into their Taylor series.
\begin{definition}
Let $C^{\omega}$ denotes a space of functions analytic near the origin, i.e.
$$
f\in C^{\omega}\qquad \Leftrightarrow \qquad \exists R>0:\quad f(x)=\sum_{n=0}^{\infty} f_n x^{n},\qquad \forall \abs{x}<R, 
$$
for some complex coefficients $f_n$.

Let $a,c\in\mathbb{C}$, so that $1-c\not\in\mathbb{N}$.
Then the \textit{hypergeometrization} is the linear operator
$$
\hypop_{c}^{a}: C^\omega \to C^\omega, 
$$ 
given by
\begin{equation}\label{naivehyp}
\hypop_{c}^a f\zav{x}:=\sum_{n=0}^\infty f_n\frac{(a)_n}{(c)_n}x^{n},
\end{equation}
where $(a)_n=a(a+1)\cdots (a+n-1)$ is the Pochhammer symbol.
\end{definition}
%
\begin{remark}
Most of the time we will make hypergeometrization with respect to the $x$ variable, or with respect to a variable which is clear from context. However, in case there is a need to stress the variable in use, we will write it in brackets like so:
$$
\hypop_{c}^a \equiv \hypop_{c}^a(x).
$$  
\end{remark}

Application of operator $\hypop_{c}^{a}$ on elementary functions can produce large number of special functions, particularly (as the name suggests) hypergeometric functions. Concretely, Gauss's hypergeometric function is trivially given by
\begin{align}
\label{2F1}\hypop_{c}^{a} (1-x)^{-b}&=\!\! \ _2 F_1\zav{\nadsebou{a\quad b}{c};x}.\\
\intertext{Similarly, we have an expression for the confluent hypergeometric function}
\label{1F1}\hypop_{c}^{a} e^x&=\!\! \ _1 F_1\zav{\nadsebou{a}{c};x},\\
\intertext{and Bessel's function}
\label{Bessel}\hypop_{c}^{\frac12} \cos(2\sqrt{x})&= \!\! \ _0 F_1\zav{\nadsebou{-}{c};-x}=\Gamma(c)x^{\frac{1-c}{2}} J_{c-1}(2\sqrt{x}).
\end{align}
In fact, as we will see in Proposition \ref{pfqrepr}, all the generalized hypergeometric functions $\!\! \ _p F_q$ can be constructed from elementary functions (by iterative application of hypergeometrization). We will also show that great number of multivariate analogues of hypergeometric functions are also images of $\hypop_{c}^{a}$. For instance Appell's functions \cite{Appell,Appell2}:
\begin{align}
\hypop_{c}^{a}(t)(1-tx)^{-b_1}(1-ty)^{-b_2}&\stackrel{(\ref{F1})}{=}F_1\zav{\nadsebou{a}{c};\nadsebou{b_1\quad b_2}{-};tx,ty}.\\
\hypop_{c_1}^{b_1}(x)\hypop_{c_2}^{b_2}(y)(1-x-y)^{-a}&\stackrel{(\ref{F2})}{=}F_2\zav{\nadsebou{a}{-};\nadsebou{b_1}{c_1}\nadsebou{b_2}{c_2};x,y}.\\
\hypop_{1}^{a_1}(x)\hypop_{\frac12}^{b_1}(x)\hypop_{1}^{a_2}(y)\hypop_{\frac12}^{b_2}(y)\hypop_{c}^{\frac32}(t)\frac{{\rm arctan}\sqrt{t^2xy-tx-ty}}{\sqrt{t^2xy-tx-ty}}&\stackrel{(\ref{F3})}{=}F_3\zav{\nadsebou{}{c};\nadsebou{a_1\quad b_1}{-}\nadsebou{a_2\quad b_2}{-};tx,ty}.\\
\hypop^{\frac12}_{c}(x)\hypop^{\frac12}_{d}(y)\hypop^{b}_{\frac12}(t)\hypop^{a}_{1}(t)
\frac{1-t(x+y)}{1-2t(x+y)+t^2(x-y)^2}&\stackrel{(\ref{F4})}{=}F_4\zav{\nadsebou{a\quad b}{-};\nadsebou{-}{c}\nadsebou{-}{d};tx,ty}.
\end{align}
But we will also deal with functions from the Horn's list $G_2$, $H_4$, $\Phi_1$, $\Phi_3$, \cite{Erdelyi}. 
\begin{remark}
All the claimed identities in this section can be checked following the link above the equality sign.
\end{remark}

\bigskip

Our main focus is the question whether there exists a ``change of variable formula'' for the operator $\hypop_c^a$. That is, is there a way how to compute hypergeometrization of a composite function in terms hypergeometrization with respect to the inner function? In symbols, we want to produce formulas of the form
$$
\hypop_{c}^a(x) f(y(x))\stackrel{?}{=} F\zav{y, \hypop_{c_j}^{a_j}(y)}f(y),
$$
where $F$ is some non-commutative expression involving $y$ and some finite number of hypergeometrzation operators $\hypop_{c_j}^{a_j}$ with various parameters.   

For some function $y$ the answer is yes. For instance, it is an easy exercise based on properties of the Pochhammer symbol that the following holds:
\begin{align}
&\hypop_c^a (x)\serovna{argscaling}\hypop_c^a (y), &  y&=S_\alpha(x):=\alpha x.\\
&\hypop_c^a (x)\stackrel{(\ref{secondpower})}{=}\hypop_{\frac{c}{2}}^{\frac{a}{2}}(y)\hypop_{\frac{c+1}{2}}^{\frac{a+1}{2}}(y), &  y&=M_2(x):=x^2.\\
&\hypop_c^a (x)\stackrel{(\ref{nthpower})}{=}\hypop_{\frac{c}{n}}^{\frac{a}{n}}(y)\hypop_{\frac{c+1}{n}}^{\frac{a+1}{n}}(y)\cdots \hypop_{\frac{c+n-1}{n}}^{\frac{a+n-1}{n}}(y) , &  y&=M_n(x):=x^n.\\
\intertext{We will show that a change of variable formula holds also for the function $x/(x-1)$
which reads:}
&\hypop^a_c\zav{x}\stackrel{(\ref{Pfaffproperty})}{=}(1-y)^{a}\hypop^a_c\zav{y}(1-y)^{-c}, & y&=P(x):=\frac{x}{x-1}.
\end{align}
The last identity -- which we call ``'Pfaff property'' -- seems to be of fundamental importance. Throughout this article we will show that this single formula is all one need to derive surprisingly large numbers of transform of special function, including:
\begin{align*}
\intertext{Pfaff transform:}
\!\! \ _2 F_1\zav{\nadsebou{a\quad b}{c};x}&\stackrel{(\ref{Pfaff})}{=}(1-x)^{-b}\!\! \ _2 F_1\zav{\nadsebou{c-a\quad b}{c};\frac{x}{x-1}}. \\
\intertext{$F_1$ transform:}
F_1\zav{\nadsebou{a}{c};\nadsebou{b_1\quad b_2}{-};x,y}&\stackrel{(\ref{F1Pfaff})}{=}
(1-x)^{-a}F_1\zav{\nadsebou{a}{c};\nadsebou{c-b_1-b_2\quad b_2}{-};\frac{x}{x-1},\frac{x-y}{x-1}}.\\
\intertext{Quadratic transform:}
\!\! \ _2 F_1\zav{\nadsebou{a\quad b}{2a};2x}&\stackrel{(\ref{2F1q})}{=}(1-x)^{-b}\!\! \ _2 F_1\zav{\nadsebou{\frac{b}{2}\quad \frac{b+1}{2}}{a+\frac12};\zav{\frac{x}{1-x}}^2}.\\
\intertext{$F_1$ to $\!\! \ _3 F_2$ reduction:}
F_1\zav{\nadsebou{b}{3a};\nadsebou{a\quad a}{-};zx,\bar z x}&\stackrel{(\ref{F1to3F2})}{=}(1-x)^{-b}\!\! \ _3 F_2\zav{\nadsebou{a\quad \frac{b}{3}\quad \frac{b+1}{3}\quad \frac{b+2}{3}}{a\quad a+\frac13\quad a+\frac23};\zav{\frac{x}{x-1}}^3},\qquad \nadsebou{z+\bar z=3}{z\bar z=3}.\\
\intertext{$F_2$ to  $\!\!\ _2 F_1$ reduction:}
F_2\zav{\nadsebou{a}{-};\nadsebou{b_1}{a}\nadsebou{b_2}{a};x,y}&\stackrel{(\ref{F2to2F1})}{=}(1-x)^{-b_1}(1-y)^{-b_2}\!\! \ _2F_1\zav{\nadsebou{b_1\quad b_2}{a};\frac{xy}{(1-x)(1-y)}}.\\
\intertext{Alternative representations for $F_1$:}
F_1\zav{\nadsebou{a}{c};\nadsebou{b_1\quad b_2}{-};x,y}&\stackrel{(\ref{F1alt})}{=}\hypop_{c-b_2}^{b_1}(x)(1-x)^{-a}\!\! \ _2 F_1\zav{\nadsebou{a\quad b_2}{c};\frac{y-x}{1-x}},\\
\end{align*}
and many more. 
Our main result is to give a change of variable formula valid for a one parameter group of functions.
\begin{theorem}\label{main} Let 
$$
y=F_m(x):=1-(1-x)^m,\qquad m\in\mathbb{Z}.
$$
Then assuming either
$$
1)\qquad m\in \szav{-2,-1,1,2},\ \forall a,c\in \mathbb{C},\qquad \text{or }\qquad 2)\qquad 
\forall m\in\mathbb{Z}\setminus \szav{0},\ a-c\in\mathbb{Z},  
$$
it holds 
\begin{equation}\label{Fnsubs}
\hypop_{c}^{a}(x)=
\zav{\frac{mx}{y}}^{1-c}(1-y)^{1+\frac{c-a}{m}}\zav{\prod_{j=1}^{m}(1-y)^{\frac{a-c-1}{m}}\hypop_{c+(j-1)\frac{a-c}{m}}^{c+j\frac{a-c}{m}}(y)}\zav{\frac{mx}{y}}^{a-1}.
\end{equation}
\end{theorem}
\begin{remark}
The product 
$
\prod_{j=1}^{m}
$
in (\ref{Fnsubs}) is understood to be naturally extended for negative $m$ and zero. Let $\szav{A_j}_{j\in\mathbb{Z}}$ be a sequence of invertible linear operators. Then we set
\begin{equation}\label{negativeproduct}
\prod_{j=1}^{0}A_j:=0,\qquad \prod_{j=1}^{-m}A_j:=\prod_{j=1}^{m}A_{1-j}^{-1},\quad \forall m\in\mathbb{N}.
\end{equation}
\end{remark}
\begin{remark}
It is the author believe that Theorem \ref{main} is not in fact limited to parameters $a$, $c$ which differs by an integer but it holds for all their (permissible) complex values. All the restrictions on $m$, $a$ and $c$ thus reflect only the author's inability to prove the theorem in full generality.    
\end{remark}
\begin{conjecture}\label{Conjecture}
The formula (\ref{Fnsubs}) holds for generic values of $a,c\in\mathbb{C}$ and all $m\in\mathbb{Z}\subset \szav{0}$.
\end{conjecture}

In summary, using Theorem \ref{main} a ``change of variable'' formula can be obtained for any function $y$ that can be written as a finite composition of 
$$
s_\alpha(x)=\alpha x,\qquad M_n(x)=x^n,\qquad  F_n(x)=1-(1-x)^n,
$$
(right now with additional restriction that $a-c\in\mathbb{Z}$). Note that $F_{-1}(x)=x/(x-1)=P(x)$. 

Here is a small sample of identities on can construct from these functions which are valid for all values of $a$ and $c$:
\begin{align}
&(1-x)^{1-c}\hypop^a_c\zav{x}(1-x)^{a-1}\stackrel{(\ref{Qt3})}{=}\hypop^{\frac{a+c-1}{2}}_c\hspace{-0.8 em}\zav{y}(1-y)^{-\frac{c-a}{2}}\hypop_{\frac{c+a-1}{2}}^{a}\hspace{-0.2 em}(y), & y&=4x(1-x).\\
&(1-x)^{c+a-1}\hypop^a_c\zav{x}(1-x)^{1-c-a}\stackrel{(\ref{Qt1monom})}{=}\hypop^{\frac{a+c-1}{2}}_c\hspace{-0.8 em}\zav{y}(1-y)^{-\frac{c-a}{2}}\hypop_{\frac{a+c-1}{2}}^{a}\hspace{-0.2 em}(y), & y&=\frac{-4x}{(1-x)^2}.\\
&(1+x)^{c+a-1}\hypop^a_c\zav{x}(1+x)^{1-c-a}\stackrel{(\ref{Qt2})}{=}\hypop^{\frac{a+c-1}{2}}_c\hspace{-0.8 em}\zav{y}(1-y)^{-\frac{c-a}{2}}\hypop_{\frac{a+c-1}{2}}^{a}(y), & y&=\frac{4x}{(1+x)^2}.\\
&(1-x)^{\frac{a}{2}}\hypop^a_c\zav{x}(1-x)^{-\frac{c}{2}}\stackrel{(\ref{Qt5})}{=}\hypop^{\frac{a}{2}}_{\frac{c+1}{2}}\zav{y}(1-y)^{-\frac{c-a}{2}}\hypop_{\frac{c}{2}}^{\frac{a+1}{2}}(y), & y&=\frac{x^2}{4(x-1)}.\\
&(1-x)^{1-c}\hypop^a_c\zav{x}(1-x)^{a-1}\\
&\stackrel{(\ref{Qt4})}{=}(1-y)^{\frac{a+c-1}{2}}\hypop^{\frac{a+c-1}{2}}_c\hspace{-0.8 em}\zav{y}(1-y)^{-\frac{c-a}{2}}\hypop_{\frac{a+c-1}{2}}^{a}(y)(1-y)^{-\frac{a+c-1}{2}}, & y&=\frac{4x(x-1)}{(1-2x)^2}.\nonumber\\
&\zav{1-\frac{x}{2}}^{a}\hypop^a_c\zav{x}\zav{1-\frac{x}{2}}^{-c}\stackrel{(\ref{Qt6})}{=}\hypop^{\frac{a}{2}}_{\frac{c+1}{2}}\zav{y}\hypop_{\frac{c}{2}}^{\frac{a+1}{2}}(y), & y&=\frac{x^2}{(2-x)^2}.\\
&(1-x^2)^{\frac{a+1}{2}}\hypop_{c}^a(x) (1-x^2)^{-\frac{c+1}{2}}\stackrel{(\ref{Qt7})}{=}\hypop_{\frac{c}{2}}^{\frac{a+1}{2}}(y)(1-y)^{-\frac{c-a}{2}}\hypop_{\frac{c+1}{2}}^{\frac{a}{2}}(y), & y&=\frac{x^2}{x^2-1}.
%
\end{align}
And so on. 

In what follows, and to demonstrate the technique, we are going to  use hypergeometrization to derive many \textit{known} identities involving special functions. There are, however, three identities which are possibly new (or at least the author is unable to find them in the literature). These are: 
\begin{itemize}
\item  A quadratic transform for $F_1$ function: Let $\beta:=\frac{a+c-1}{2}$. Then
\begin{equation}\label{F1Qi}
F_1\zav{\nadsebou{a}{c};\nadsebou{\beta \quad \beta }{-};\tau_+ x,\tau_- x}\serovna{F1Q}(1-x)^{-2\beta} F_1\zav{\nadsebou{\beta}{c};\nadsebou{\frac{c-a}{2}\quad a}{-};\frac{-4x}{(1+x)^2},\frac{-4x t}{(1+x)^2}},
\end{equation}
where
$$
\tau_{\pm}:=2\zav{(2t-1)^2\pm\sqrt{t(t-1)}}.
$$
\item A semi-cubic reduction of $F_1$ to $\!\!\ _2 F_1$:
\begin{equation}
(1-x)^{-2a}\!\! \ _2 F_1\zav{\nadsebou{\frac{a}{3}\quad \frac{2a}{3}}{\frac{a}{3}+1};\zav{\frac{x}{x+1}}^3}\serovna{F1semicubic}F_1\zav{\nadsebou{a}{a+1};\nadsebou{\frac12\quad \frac23 a}{-};4x(1-x),3x(1-x)}.
\end{equation}
\item $G_2$ to  $F_2$ conversion:
\begin{equation}
G_2\zav{a\quad c; \nadsebou{b_1\quad b_2}{-};x,y}\stackrel{(\ref{G2toF2})}{=}(1+x)^{-b_1}(1+y)^{-b_2}F_2\zav{\nadsebou{1-c-a}{-};\nadsebou{b_1}{1-c}\nadsebou{b_2}{1-a};\frac{x}{x+1},\frac{y}{y+1}}.
\end{equation}
\end{itemize}
Particularly, it does not seem to be possible to derive the first formula (\ref{F1Qi}) from Carlson's results about quadratic transforms of $F_1$ function given in \cite{Carlson2}.

The structure if the paper is as follows: Basic properties of hypergeometrization operator are discussed in Section \ref{BPsec}. In Section \ref{S5} the methodology of representing a special functions via hypergeometrization is described. Section \ref{Pfaffpropertysec} introduces the Pfaff property. Its consequences are discussed in Section \ref{S8}. Treatment of the change of variable formula is done in Section \ref{S10}. Finally, in Section \ref{Sproof} we prove Theorem \ref{main} and provide some supporting evidence for Conjecture \ref{Conjecture}.  
\begin{remark}
The concept of hypergeometrization was introduced by the present author in \cite{B2} and was also mentioned in \cite{B7}. It can be understood as a Hadamard product (or a convolution)
$$
 \hypop_{c}^{a}f(x)=\!\! \ _{2} F_1\zav{\nadsebou{a \quad 1}{c};x}\star f(x), 
$$
where the Hadamard product of the two formal power series $g(x)=\sum_{k\geq 0} g_k x^k$, $h(x)=\sum_{k\geq 0} h_k x^k$ is defined 
$$
g(x)\star h(x):=\sum_{k=0}^{\infty}g_k h_k x^k.
$$
Before \cite{B2}, a linear operator which brings a function to its Hadamard product with some hypergeometric function (i.e. to its hypergeometrization) appeared also in \cite{Carlson} and elsewhere. But hypergeometrization is a special case of Hadamard product, and -- as we will endeavor to show -- has many properties the general Hadamard product does not posses.
\end{remark}

\section{Basic properties}\label{BPsec}
An important property of hypergeometrization is that (generically) it does not change the radius of convergence. 
\begin{proposition}\label{Radius}
Let $R>0$ be a radius of convergence of the following power series: 
$$
f(x)=\sum_{n=0}^\infty f_n x^n,\qquad \abs{x}<R.
$$
Let $1-a,1-c\not\in\mathbb{N}$. Then
$$
\hypop_{c}^a f(x)=\sum_{n=0}^\infty \frac{(a)_n}{(c)_n}f_n x^n,
$$
converges for all $\abs{x}<R$.
\end{proposition}
\begin{proof}
It is a standard result for $\Gamma$ function that
$$
\lim_{n\to \infty }n^{c-a}\frac{(a)_n}{(c)_n}=\lim_{n\to \infty }n^{c-a}\frac{\Gamma(a+n)\Gamma(c)}{\Gamma(a)\Gamma(c+n)}= \frac{\Gamma(c)}{\Gamma(a)},
$$
and thus the introduced factor $(a)_n/(c)_n$ grows only polynomially in $n$ and is therefore negligible comparing to the exponential behavior of $x^n$ term.
\end{proof}
Another crucial observation for our purposes is that when the parameters $a$, $c$ differ by an integer, the hypergeometrization reduces to a differential operator. 
\begin{equation}\label{Hdifop}
\hypop_{a}^{a+n}(x)=\frac{(a+x\partial_x)_n}{(a)_n}.
\end{equation}
The proof is straightforward.

Some additional elementary properties of hypergeometrization includes:
\begin{align}
&\hypop^{a}_{c}\zav{\alpha f+\beta g}=\alpha\hypop^{a}_{c} f+\beta \hypop^{a}_{c} g, & \text{linearity,}\label{linearity}\\
&\hypop^{a}_{c} \hypop^{b}_{d}=\hypop^{b}_{d}\hypop^{a}_{c}, & \text{commutativity,}\label{commutativity}\\
&\hypop^{a}_{c} \hypop^{b}_{d}=\hypop^{a}_{d}\hypop^{b}_{c}=\hypop^{b}_{c} \hypop^{a}_{d}, & \text{parameter exchange,}\label{exchange}\\
&\zav{\hypop^{a}_{c}}^{-1}=\hypop^{c}_{a}, & \text{inverse,}\label{inverse}\\
&\hypop^{a}_{c} x^n =\frac{(a)_n}{(c)_n}x^n\hypop^{a+n}_{c+n}, & \text{shift,}\label{shift}\\
&(\partial_x)^n\hypop^{a}_{c} =\frac{(a)_n}{(c)_n}\hypop^{a+n}_{c+n}(\partial_x)^n, & \text{dual shift,}\label{coshift}\\
&\hypop^{a}_c(\alpha x)=\hypop^{a}_c(x), &\text{argument scaling,}\label{argscaling}\\
&\hypop^{a}_c(x)=\hypop^{\frac{a}{2}}_{\frac{c}{2}}\zav{x^2}\hypop^{\frac{a+1}{2}}_{\frac{c+1}{2}}\zav{x^2},&\text{argument square,}\label{secondpower}\\
&\hypop^{a}_c(x)=\hypop^{\frac{a}{n}}_{\frac{c}{n}}\zav{x^n}\dots\hypop^{\frac{a+n-1}{n}}_{\frac{c+n-1}{n}}\zav{x^n}&\text{$n$-th power,}\label{nthpower}\\
&c\hypop^{a}_{c}-a \hypop^{a+1}_{c+1}+(a-c)\hypop^{a}_{c+1}=0, & \text{contiguous relation,}\label{contiguous}\\
&\hypop^{a}_{a+1} \hypop^{-a}_{1-a}=\frac12\hypop^{a}_{a+1}+\frac12\hypop^{-a}_{1-a}, & \text{per partes.}\label{perpartes}
\end{align}
Here the function $f$, $g$ are analytic near the origin, $\alpha,\ \beta \in\C$ and $n\in\N$. Parameters $a,\ b, \ c,\ d$ can be arbitrary complex numbers with the possible restriction on the lower parameters $1-c\not\in \N$.
\begin{proof}
Since we are working on function analytic near origin, it is actually sufficient to verify all these claims only on monomials $x^n$ which is -- mostly -- straightforward and are left to the reader as an stimulating exercise. Identities (\ref{secondpower}), (\ref{nthpower}) are based on the following property of Pochhammer symbols:
\begin{equation}\label{Pochpower}
\forall n,k\in \mathbb{N}:\qquad (a)_{nk}=\zav{\frac{a}{n}}_k\zav{\frac{a+1}{n}}_k\cdots \zav{\frac{a+n-1}{n}}_k n^{nk}.
\end{equation}
A property that perhaps deserves some comment is the very last one. It too can be very easily checked on monomials as follows:
$$
\hypop^{a}_{a+1} \hypop^{-a}_{1-a}x^n=\frac{(a)_n(-a)_n}{(a+1)_n(1-a)_n}x^n=\frac{-a^2}{(a+n)(n-a)}x^n=\frac{a}{2(n+a)}x^n-\frac{a}{2(n-a)}x^n=\frac12\hypop_{1-a}^{-a}x^n+\frac12\hypop_{1+a}^{a}x^n.
$$
But why is it called ``per partes''?

Remember that from (\ref{Hdifop}) when the upper parameter differs from the lower one by 1, the hypergeometrization reduces to:
$$
\hypop^{a+1}_{a}=\frac{a+x\partial_x}{a}=\frac{1}{a} x^{1-a}\partial_x x^{a}.
$$ 
Thus its inverse is an integral operator
$$
\hypop^{a}_{a+1}=\zav{\hypop^{a+1}_{a}}^{-1}=\zav{\frac{1}{a} x^{1-a}\partial_x x^{a}}^{-1}=ax^{-a}\int \dd x x^{a-1},
$$
modulo integration constant, of course.
Hence
\begin{align*}
\hypop^{a}_{a+1}\hypop^{-a}_{1-a}&=ax^{-a}\int \dd x x^{a-1}(-a)x^{a}\int \dd x x^{-a-1}=-a^2 x^{-a}\int \dd x x^{2a-1}\int \dd x x^{-a-1}\\
&=-a^2x^{-a}\zav{\frac{x^{2a}}{2a}\int \dd x x^{-a-1}-\int \dd x \frac{x^{2a}}{2a}\partial_x\int \dd x x^{-a-1}}\\
&=\frac{-a}{2}x^{a}\int \dd x x^{-a-1}+\frac{a}{2}x^{-a}\int \dd x x^{a-1}=\frac12\hypop_{1-a}^{-a}+\frac12\hypop_{1+a}^{a}.
\end{align*}
Here we have used ``integration per partes'' in the operator notation:
$$
\int \dd x x^{2a-1}= \frac{x^{2a}}{2a}-\int \dd x \frac{x^{2a}}{2a}\partial_x.
$$
\end{proof}
\section{Special function representation}\label{S5}
\subsection{Generalized hypergeometric functions}
Remember:
\begin{definition} \textit{Generalized hypergeometric functions} $\!\! \ _p F_q$ are defined as follows:
\begin{equation}\label{seriesdef}
\!\! \ _p F_q\zav{\nadsebou{a_1\dots a_p  }{c_1\dots c_q};x}:=\sum_{k=0}^{\infty}\frac{(a_1)_k\cdots (a_p)_k}{(c_1)_k\cdots (c_q)_k}\frac{x^k}{k!},\qquad 1-c_k\not\in \mathbb{N}, \forall k.
\end{equation}
The series converges in the entire complex plane if $p\leq q$. For $p=q+1$ it converges in the unit disc $\abs{x}<1$ and for $p>q+1$ it is generally divergent unless one of the upper parameters is a negative integer, in which case the series terminates and the resulting hypergeometric function is actually a polynomial.
\end{definition}  
\begin{proposition}\label{pfqrepr}For $n\in\mathbb{N}$ let 
\begin{equation}\label{fndef}
f_n(x):=\frac{1}{n}\zav{e^{n z_0 \sqrt[n]{x}}+e^{n z_1 \sqrt[n]{x}}+\dots e^{n z_{n-1} \sqrt[n]{x}}}=\sum_{k=0}^{\infty}\frac{n^{nk}x^k}{(nk)!},\qquad z_j:=e^{\frac{2\pi\imag j}{n}}.
\end{equation}
In particular
\begin{align*}
f_1&=e^x,\\
f_2&=\frac12\zav{e^{\sqrt{x}}+e^{-\sqrt{x}}}=\cosh(2\sqrt{x}),\\
f_3&=\frac13\zav{e^{3\sqrt[3]{x}}+2e^{-\frac32\sqrt[3]{x}}\cos\zav{\frac{3\sqrt{3}}{2}\sqrt[3]{x}}},\\
&\vdots
\end{align*}
Then for any complex numbers $a_1,\dots, a_m$  and $c_1,\dots c_{m+n-1}\in\mathbb{C}$, such that $1-c_i\not\in\mathbb{N}$ $\forall i$ it holds:
$$
\!\! \ _m F_{m+n-1}\zav{\nadsebou{a_1\dots a_m}{c_1\dots c_{n+m-1}};x}=
\hypop_{c_1}^{\frac{1}{n}} \hypop_{c_2}^{\frac{2}{n}}\dots \hypop_{c_{n-1}}^{\frac{n-1}{n}} \hypop_{c_n}^{a_1}\dots \hypop_{c_{n+m-1}}^{a_m} f_n(x). 
$$
\end{proposition}
\begin{proof} From (\ref{Pochpower}) it follows that:
$$
(nk)!=(1)_{nk}=\zav{\frac{1}{n}}_{k}\zav{\frac{2}{n}}_{k}\cdots \zav{\frac{n-1}{n}}_k k! n^{nk}.
$$
Thus
$$
f_n=\sum_{k=0}^{\infty}\frac{n^{nk}x^k}{(nk)!}=\!\! \ _0 F_{n-1}\zav{\nadsebou{-}{\frac{1}{n}\quad \frac{2}{n}\dots \frac{n-1}{n}};x}.
$$
The result is obtained by successive application of definition (\ref{naivehyp}). 
\end{proof} 
The one advantage of this approach is that it makes questions of convergence clear.  Since, evidently, the hypergeometrization does not change the region of convergence, we can see at once that the series $\!\! \ _{q+1} F_q$ converges in the unit disk (since those functions originated from $(1-x)^{-b}$) and the rest $\!\! \ _{p} F_q$ $(p\leq q)$ converges everywhere since they are constructed from entire functions like $e^x,\cosh(2\sqrt{x})$ etc.
\subsection{Appell's functions.}
Appell's function are defined by the following double series:
\begin{align}
F_1\zav{\nadsebou{a}{c};\nadsebou{b_1\quad b_2}{-};x,y}&:=\sum_{j,k=0}^{\infty}\frac{(a)_{j+k}}{(c)_{j+k}}\frac{(b_1)_j(b_2)_k}{j!k!} x^j y^k,\\
F_2\zav{\nadsebou{a}{-};\nadsebou{b_1}{c_1}\nadsebou{b_2}{c_2};x,y}&:=\sum_{j,k=0}^{\infty}\frac{(a)_{j+k}}{j!k!}\frac{(b_1)_j(b_2)_k}{(c_1)_j(c_2)_k} x^j y^k,\\
F_3\zav{\nadsebou{-}{c};\nadsebou{a_1\quad b_1}{-}\nadsebou{a_2\quad b_2}{-};x,y}&:=\sum_{j,k=0}^{\infty}\frac{(a_1)_j(b_1)_j(a_2)_k(b_2)_k}{(c)_{j+k}j!k!} x^j y^k,\\
F_4\zav{\nadsebou{a\quad b}{-};\nadsebou{-}{c\quad d};x,y}&:=\sum_{j,k=0}^{\infty}\frac{(a)_{j+k}(b)_{j+k}}{j!k!(c)_{j}(d)_k} x^j y^k.
\end{align} 
All of these functions can be as well represented as a hypergeometrization of some elementary function:
\begin{proposition}
\begin{align}
\intertext{Appell's $F_1$ function:}
\hypop_{c}^{a}(t)(1-tx)^{-b_1}(1-ty)^{-b_2}&=F_1\zav{\nadsebou{a}{c};\nadsebou{b_1\quad b_2}{-};tx,ty}.\label{F1}\\
\intertext{Appell's $F_2$ function:}
\hypop_{c_1}^{b_1}(x)\hypop_{c_2}^{b_2}(y)(1-x-y)^{-a}&=F_2\zav{\nadsebou{a}{-};\nadsebou{b_1}{c_1}\nadsebou{b_2}{c_2};x,y}. \label{F2}\\
\intertext{Appell's $F_3$ function.}
\hypop_{1}^{a_1}(x)\hypop_{\frac12}^{b_1}(x)\hypop_{1}^{a_2}(y)\hypop_{\frac12}^{b_2}(y)\hypop_{c}^{\frac32}(t)\frac{{\rm arctan}\sqrt{t^2xy-tx-ty}}{\sqrt{t^2xy-tx-ty}}&=F_3\zav{\nadsebou{}{c};\nadsebou{a_1\quad b_1}{-}\nadsebou{a_2\quad b_2}{-};tx,ty}.\label{F3}\\
\intertext{Appell's $F_4$ function.}
\hypop^{\frac12}_{c}(x)\hypop^{\frac12}_{d}(y)\hypop^{b}_{\frac12}(t)\hypop^{a}_{1}(t)
\frac{1-t(x+y)}{1-2t(x+y)+t^2(x-y)^2}&=F_4\zav{\nadsebou{a\quad b}{-};\nadsebou{-}{c}\nadsebou{-}{d};tx,ty}.\label{F4}
\end{align} 
\end{proposition}
\begin{proof}
The proof amounts to show that
\begin{align*}
(1-tx)^{-b_1}(1-ty)^{-b_2}&=F_1\zav{\nadsebou{c}{c};\nadsebou{b_1\quad b_2}{-};tx,ty},\\
(1-x-y)^{-a}&=F_2\zav{\nadsebou{a}{-};\nadsebou{c_1}{c_1}\nadsebou{c_2}{c_2};x,y}, \\
\frac{{\rm arctan}\sqrt{t^2xy-tx-ty}}{\sqrt{t^2xy-tx-ty}}&=F_3\zav{\nadsebou{}{\frac32};\nadsebou{1\quad \frac12}{-}\nadsebou{1\quad \frac12}{-};tx,ty},\\ \frac{1-t(x+y)}{1-2t(x+y)+t^2(x-y)^2}&=F_4\zav{\nadsebou{1\quad \frac12}{-};\nadsebou{-}{\frac12}\nadsebou{-}{\frac12};tx,ty},
\end{align*}
which is left to the reader as an easy exercise.
\end{proof}
Once again we can retrieve the information about the regions of convergence for Appell's series from their elementary origins. Since the hypergeometrization does not change the radius of convergence, we can deduce from the fact that
$$
(1-x)^{-b_1}(1-y)^{-b_2}=\sum_{j,k=0}^{\infty}\frac{(b_1)_j(b_2)_k}{j!k!}x^j y ^k<\infty \qquad \Leftarrow \qquad \abs{x}<1,\abs{y}<1,
$$
that the same is true for $F_1$ function.

Similar arguments in other cases gives us the following overall list:
\begin{align*}
F_1:& & \abs{x}<1,\abs{y}&<1,\\
F_2:& & \abs{x+y}&<1,\\
F_3:& & \abs{xy-x-y}&<1,\\
F_4:& & \abs{\sqrt{x}+\sqrt{y}}<1,\abs{\sqrt{x}-\sqrt{y}}&<1.
\end{align*}
This trick is, essentially, \textit{Horn's principle} in reverse. 

(Horn's principle states that the region of convergence of any hypergeometric function does not depend on the specific values of parameters -- safe for some exceptional pathological values, like negative integers and so on. See \cite{Horn}.)

\begin{example}\label{F1to2F1ex}
The approach of hypergeometrization helps to understand some of the various transforms valid for these functions. For example, equating $x=y=1, t=1$ in the formula for $F_1$ function (\ref{F1}) we obtain
$$
F_1 \zav{\nadsebou{a}{c};\nadsebou{b_1\quad b_2}{-};x,x}=\!\! \ _2 F_1\zav{\nadsebou{a\quad b_1+b_2}{c};x},
$$
since 
$$
(1-x)^{-b_1}(1-x)^{-b_2}=(1-x)^{-(b_1+b_2)}.
$$
\end{example}
\begin{example}
Similarly, from the fact that
$$
(1-x)^{-b}(1+x)^{-b}=(1-x^2)^{-b},
$$
we can easily deduce using (\ref{Pochpower})
$$
F_1 \zav{\nadsebou{a}{c};\nadsebou{b\quad b}{-};x,-x}=\!\! \ _3 F_2\zav{\nadsebou{\frac{a}{2}\quad \frac{a+1}{2}\quad b}{\frac{c}{2}\quad \frac{c+1}{2}};x^2}.
$$
\end{example}
\begin{example}
The following elementary identity
\begin{equation}\label{elem}
(1-x-y)^{-a}=(1-x)^{-a}\zav{1-\frac{y}{1-x}}^{-a},
\end{equation}
implies a representation of Appell's $F_2$ function in the form
\begin{equation}\label{F2repr}
F_2\zav{\nadsebou{a}{-};\nadsebou{b_1}{c_1}\nadsebou{b_2}{c_2};x,y}=
\hypop_{c_1}^{b_1}(x)(1-x)^{-a}\!\!\ _2 F_1\zav{\nadsebou{a\quad b_2}{c_2};\frac{y}{1-x}}.
\end{equation}
The argument is as follows:
\begin{align*}
F_2\zav{\nadsebou{a}{-};\nadsebou{b_1}{c_1}\nadsebou{b_2}{c_2};x,y}&\stackrel{(\ref{F2})}{=}
\hypop_{c_1}^{b_1}(x)\hypop_{c_2}^{b_2}(y) (1-x-y)^{-a}\stackrel{(\ref{elem})}{=}
\hypop_{c_1}^{b_1}(x)\hypop_{c_2}^{b_2}(y)
(1-x)^{-a}\zav{1-\frac{y}{1-x}}^{-a}\\
&\stackrel{(\ref{linearity})}{=}
\hypop_{c_1}^{b_1}(x)(1-x)^{-a}  \hypop_{c_2}^{b_2}(y)\zav{1-\frac{y}{1-x}}^{-a}\\
&\stackrel{(\ref{argscaling})}{=}\hypop_{c_1}^{b_1}(x)(1-x)^{-a}  \hypop_{c_2}^{b_2}\zav{\frac{y}{1-x}}\zav{1-\frac{y}{1-x}}^{-a}\\
&\stackrel{(\ref{2F1})}{=}\hypop_{c_1}^{b_1}(x)(1-x)^{-a} \!\! \ _2 F_1\zav{\nadsebou{a\quad b_2}{c_2};\frac{y}{1-x}}.
\end{align*}
The question of when (\ref{F2repr}) holds is not trivial. But it perhaps worth noting that, in some sense, the equality (\ref{F2repr}) \textit{should be} valid whenever (\ref{elem}) is. We will not endeavor to make this statement precise.
 \end{example}
\begin{example}
Likewise, we can ask what relation of special functions is induced by the following elementary identity
\begin{equation}\label{elem2}
(1-x)^{-b_1}(1-y)^{-b_2}=\zav{\frac{y}{x}}^{-b_2}(1-x)^{-b_1-b_2}\zav{1-\frac{1-\frac{x}{y}}{1-x}}^{-b_2}.
\end{equation}
Changing the variables to $x\to tx$, $y\to ty$ we obtain

\begin{align*}
(1-tx)^{-b_1}(1-ty)^{-b_2}&=\zav{\frac{y}{x}}^{-b_2}(1-tx)^{-b_1-b_2}\zav{1-\frac{1-\frac{x}{y}}{1-tx}}^{-b_2}.\\
\intertext{Applying $\hypop_{c}^{a}(t)$ to both sides yields}
F_1\zav{\nadsebou{a}{c};\nadsebou{b_1\quad b_2 }{-};tx,ty}&=
\zav{\frac{y}{x}}^{-b_2}\hypop_{c}^{a}(t)(1-tx)^{-b_1-b_2}\zav{1-\frac{1-\frac{x}{y}}{1-tx}}^{-b_2}.\\
&=\zav{\frac{y}{x}}^{-b_2}\hypop_{c}^{a}(t)(1-tx)^{-b_1-b_2}\!\! \ _2 F_1\zav{\nadsebou{b_1+b_2\quad b_2}{b_1+b_2};\frac{1-\frac{x}{y}}{1-tx}}.\\
&\stackrel{(\ref{F2repr})}{=}\zav{\frac{y}{x}}^{-b_2}
F_2\zav{\nadsebou{b_1+b_2}{-};\nadsebou{a}{c}\nadsebou{b_2}{b_1+b_2};tx,1-\frac{x}{y}}.
\end{align*}
Altogether we find the following known relationship between Appell's $F_1$ and $F_2$ function:
\begin{equation}\label{F1toF2}
F_1\zav{\nadsebou{a}{c};\nadsebou{b_1\quad b_2}{-};x,y}=\zav{\frac{y}{x}}^{-b_2}F_2\zav{\nadsebou{b_1+b_2}{-};\nadsebou{a}{c}\nadsebou{b_2}{b_1+b_2};x,1-\frac{x}{y}}.
\end{equation}
\end{example}
\subsection{Horn's functions}
Similarly, we can deal with other multi-variable hypergeometric function. Including the Appell's functions there are altogether 28 function on Horn's list (see \cite{Erdelyi}). $G$-family of functions is defined as follows:
\begin{align}
G_1\zav{\nadsebou{a}{-};b_1\quad b_2;x,y}&:=\sum_{j,k=0}^{\infty}\frac{(a)_{j+k}}{j!k!}(b_1)_{j-k}(b_2)_{k-j}x^j y^k,\\
G_2\zav{a\quad c;\nadsebou{b_1\quad b_2}{-};x,y}&:=\sum_{j,k=0}^{\infty}(a)_{j-k}(c)_{k-j}\frac{(b_1)_{j}(b_2)_{k}}{j!k!}x^j y^k,\\
G_3\zav{a\quad c;x,y}&:=\sum_{j,k=0}^{\infty}\frac{(a)_{2j-k}(c)_{2k-j}}{j!k!}x^j y^k.
\end{align}
We are able to give a representation for $G_2$:
\begin{proposition}\label{G2prop}
For generic values of $a,c,b_1,b_2\in\mathbb{C}$ it holds:
\begin{equation}\label{G2repr}
G_2\zav{a\quad c;\nadsebou{b_1\quad b_2}{-};x,y}=\hypop_{1-c}^{b_1}(x)\hypop_{1-a}^{b_2}(y) (1+y)^{-c}(1+x)^{-a}(1-xy)^{c+a-1}.
\end{equation}
Therefore the double sum $G_2$ converges for
$$
\abs{y}<1,\qquad \abs{x}<1,\qquad \abs{xy}<1,
$$
\end{proposition}
\begin{proof}
To prove hypergeometric representation of $G_2$ and also its region of convergence, all we have to do is to show that
\begin{equation}\label{toshow}
G_2\zav{a\quad c;\nadsebou{1-c\quad 1-a}{-};x,y}=(1+y)^{-c}(1+x)^{-a}(1-xy)^{c+a-1}.
\end{equation}
Starting with
$$
G_2\zav{a\quad c;\nadsebou{1-c\quad 1-a}{-};x,y}=\sum_{j,k=0}^{\infty}(a)_{j-k}(c)_{k-j}\frac{(1-c)_j(1-a)_k}{j!k!}x^j y^k,
$$
and using the identities
$$
(a)_{j-k}=\frac{(a)_j}{(1-a-j)_k}(-1)^k,\qquad (c)_{k-j}=\frac{(-1)^{j}(c-j)_k}{(1-c)_j},
$$
we obtain
$$
=\sum_{j,k=0}^{\infty}\frac{(a)_j}{j!}\frac{(c-j)_k (1-a)_k}{(1-a-j)_kk!}(-x)^j (-y)^k=
\sum_{j=0}^{\infty}\frac{(a)_j}{j!}(-x)^j\!\! \ _2 F_1\zav{\nadsebou{c-j\quad 1-a}{1-a-j};-y}
$$
$$
\stackrel{(\ref{Euler})}{=}(1+y)^{-c}\sum_{j=0}^{\infty}\frac{(a)_j}{j!}(-x)^j\!\! \ _2 F_1\zav{\nadsebou{1-a-c\quad -j}{1-a-j};-y}=(1+y)^{-c}
\sum_{j,k=0}^{\infty}\frac{(a)_j}{(j-k)!}(-x)^j\frac{(1-a-c)_k}{(1-a-j)_k k!}y^k
$$
rearranging the terms $j\to j+k$ we obtain
$$
=(1+y)^{-c}
\sum_{j,k=0}^{\infty}\frac{(a)_{j}}{j!}(-x)^{j+k}\frac{(1-a-c)_k}{k!}(-y)^k=(1+y)^{-c}(1+x)^{-a}(1-xy)^{a+c-1}.
$$
\end{proof}  
The function $G_1$ can be represented via the following link with the $F_4$ function:
\begin{proposition}
$$
G_1\zav{\nadsebou{a}{-};b_1\quad b_2;x,y}=(1+x+y)^{-a}F_4\zav{\nadsebou{a\quad 1-b_1-b_2}{-};\nadsebou{-}{1-b_1\quad 1-b_2};\frac{y}{1+x+y},\frac{x}{1+x+y}},
$$
\end{proposition}
which we state without proof only as a curiosity. At the moment the author is not aware of any simple representation of the $G_3$ functions.

There are more functions from the Horn's list that have very nice representation, namely the $H_4$ function and functions $\Phi_1$, $\Phi_2$, $\Phi_3$ defined as
\begin{align}
H_4\zav{a;\nadsebou{-}{c}\nadsebou{b}{d};x,y}&:=\sum_{j,k=0}^{\infty}\frac{(a)_{2j+k}}{j!k!}\frac{(b)_{k}}{(c)_{j}(d)_k}x^j y^k,\\
\Phi_1\zav{\nadsebou{a}{c};\nadsebou{-}{-}\nadsebou{b}{-};x,y}&:=\sum_{j,k=0}^{\infty}\frac{(a)_{j+k}(b)_{j}}{(c)_{j+k}j!k!}x^j y^k,\\
\Phi_2\zav{\nadsebou{-}{c};\nadsebou{b_1\quad b_2}{-};x,y}&:=\sum_{j,k=0}^{\infty}\frac{(b_1)_{j}(b_2)_{k}}{(c)_{j+k}j!k!}x^j y^k,\\
\Phi_3\zav{\nadsebou{-}{c};\nadsebou{b}{-}\nadsebou{-}{-};x,y}&:=\sum_{j,k=0}^{\infty}\frac{(b)_{j}}{(c)_{j+k}j!k!}x^j y^k.
\end{align}
\begin{proposition}\label{H4prop}
For generic values of parameters it holds:
\begin{align}
\label{H4repr}\hypop_{d}^b(y)\hypop_{c}^{\frac{a+1}{2}}(x)\zav{(1-y)^2-4x}^{-\frac{a}{2}}&=H_4\zav{a;\nadsebou{-}{c}\nadsebou{b}{d};x,y}.\\
\hypop_{c}^{a}(t)e^{tx}(1-ty)^{-b}&
\label{Phi1repr}=\Phi_1\zav{\nadsebou{a}{c};\nadsebou{-}{-}\nadsebou{b}{-};tx,ty}.\\
\label{Phi2repr}\hypop_{c}^{c-b_2}(t)\hypop_{c-b_2}^{b_1}(x)e^{t(x-y)}&=e^{-ty}\Phi_2\zav{\nadsebou{-}{c};\nadsebou{b_1\quad b_2}{-};tx,ty}.\\
\label{Phi3repr}\hypop_{c}^{c-b}(t)\hypop_{c-b}^{\frac12}(y)\cosh(2\sqrt{ty})e^{-tx}&=e^{-tx}\Phi_3\zav{\nadsebou{-}{c};\nadsebou{b}{-}\nadsebou{-}{-};tx,ty}.
\end{align}
\end{proposition}
\begin{proof} For the first three representations it suffices to establish the following special cases:
\begin{align*}
H_4\zav{a;\nadsebou{-}{\frac{a+1}{2}}\nadsebou{d}{d};x,y}&=\zav{(1-y)^2-4x}^{-\frac{a}{2}}.\\
\Phi_1\zav{\nadsebou{c}{c};\nadsebou{-}{-}\nadsebou{b}{-};tx,ty}&=e^{tx}(1-ty)^{-b}.\\
e^{-ty}\Phi_2\zav{\nadsebou{-}{b_1+b_2};\nadsebou{b_1\quad b_2}{-};tx,ty}&=\!\! \ _1 F_1\zav{\nadsebou{b_1}{b_1+b_2};t(x-y)},
\end{align*}
which are left to the reader.
The last representation can be proved as follows:
\begin{align*}
e^{-tx}\Phi_3\zav{\nadsebou{-}{c};\nadsebou{b}{-}\nadsebou{-}{-};tx,ty}&=e^{-tx}\sum_{j,k} \frac{(b)_j}{(c)_{j+k}}\frac{(tx)^j (ty)^k}{j!k!}= e^{-tx}\sum_{k} \frac{(ty)^k}{(c)_{k} k!}\!\! \ _1 F_1\zav{\nadsebou{b}{c+k};xt}.\\
\intertext{Using the well known Kummer transform
\begin{equation}\label{Kummer}
\!\! \ _1 F_1\zav{\nadsebou{a}{c};x}=e^{x}\!\! \ _1 F_1\zav{\nadsebou{c-a}{c};-x},
\end{equation}
we obtain}
e^{-tx}\Phi_3\zav{\nadsebou{-}{c};\nadsebou{b}{-}\nadsebou{-}{-};tx,ty}&\stackrel{(\ref{Kummer})}{=}\sum_{k} \frac{(ty)^k}{(c)_{k} k!}\!\! \ _1 F_1\zav{\nadsebou{c-b+k}{c+k};-xt}
\serovna{1F1}\sum_{k} \frac{(ty)^k}{(c)_{k} k!}\hypop_{c+k}^{c-b+k}(t)e^{-xt}
\\
&\serovna{shift} \sum_{k} \hypop_{c}^{c-b}(t) \frac{(yt)^k}{(c-b)_{k} k!} e^{-xt}
=\hypop_{c}^{c-b}(t) \hypop_{c-b}^{\frac12}(y)\sum_{k}  \frac{(yt)^k}{\zav{\frac12}_{k} k!} e^{-xt}\\
&=\hypop_{c}^{c-b}(t) \hypop_{c-b}^{\frac12}(y)\cosh\zav{2\sqrt{yt}}e^{-xt}.
\end{align*}

\end{proof}

\section{Pfaff property}\label{Pfaffpropertysec}
\begin{proposition}
Let 
$$
y(x):=\frac{x}{x-1}.
$$ 
Then
\begin{equation}\label{Pfaffproperty}
(1-x)^{a}\hypop_{c}^{a}(x) (1-x)^{-c}=\hypop_{c}^a(y).
\end{equation}
\end{proposition}
\begin{proof}
Clearly, it is enough to check the claim on monomials.
\begin{align*}
(1-x)^{a}\hypop_{c}^{a} (1-x)^{-c} y(x)^n&=(1-x)^{a}\hypop_{c}^{a} (-x)^n (1-x)^{-{c+n}}
\serovna{shift} (1-x)^{a}(-x)^n\frac{(a)_n}{(c)_n}\hypop_{c+n}^{a+n} (1-x)^{-{c+n}}\\
&\serovna{2F1}(1-x)^{a}(-x)^n\frac{(a)_n}{(c)_n}\!\! \ _2 F_1\zav{\nadsebou{c+n\quad a+n}{c+n};x}\\
&=(1-x)^{a}(-x)^n\frac{(a)_n}{(c)_n}(1-x)^{-a-n}=\frac{(a)_n}{(c)_n}y^n.
\end{align*}
\end{proof}

\begin{example}\label{Pfaffex1} 
A consequence of the following elementary identity
$$
(1-x)^{-b}=(1-x)^{-c}\zav{1+\frac{x}{1-x}}^{b-c}=(1-x)^{-c}(1-y)^{b-c},\qquad y:=\frac{x}{x-1},
$$
is a well known identity called ``Pfaff transform'' \cite[15.8.1]{dlmf15.8}:
\begin{equation}\label{Pfaff}
\!\! \ _2 F_1\zav{\nadsebou{a \quad b}{c};x}=(1-x)^{-a}\!\! \ _2 F_1\zav{\nadsebou{a \quad c-b}{c};\frac{x}{x-1}}.\qquad \text{(Pfaff transform.)}
\end{equation} 
The argument is as follows:
\begin{align*}
\!\! \ _2 F_1\zav{\nadsebou{a\quad b}{c};x}&\serovna{2F1}\hypop_c^{a}(1-x)^{-b}
=\hypop_c^{a}(1-x)^{-c}\zav{1+\frac{x}{1-x}}^{b-c}\\
&\serovna{Pfaffproperty}(1-x)^{-a}\hypop_c^a(y)(1-y)^{b-c}\serovna{2F1}(1-x)^{-a}\!\! \ _2 F_1\zav{\nadsebou{a\quad c-b}{c};y}\\
&=(1-x)^{-a}\!\! \ _2 F_1\zav{\nadsebou{a\quad c-b}{c};\frac{x}{x-1}}.
\end{align*}
Notice that this transform applied twice lead back to the original function. In other words, the Pfaff transform is an involution. There is an additional obvious involution related to the fact that the function $\!\! \ _2 F_1$ is symmetrical with respect to the upper parameters $a$, $b$:
\begin{equation}\label{Swap}
\!\! \ _2 F_1\zav{\nadsebou{a \quad b}{c};x}=\!\! \ _2 F_1\zav{\nadsebou{b \quad a}{c};x}.\qquad \text{(Parameter swap.)}
\end{equation}
If we combine these -- i.e. we first perform Pfaff transforms, then swap the upper parameters and then Pfaff transform again, we discover new identity, called ``Euler transform'' \cite[15.8.1]{dlmf15.8}:
\begin{equation}\label{Euler}
\!\! \ _2 F_1\zav{\nadsebou{a \quad b}{c};x}=(1-x)^{c-a-b}\!\! \ _2 F_1\zav{\nadsebou{c-a \quad c-b}{c};x}.\qquad \text{(Euler transform.)}
\end{equation}

\end{example}
\begin{example}\label{Pfaffex2}
The same argument can be used to derive similar transform for the $F_1$ Appell's function. Starting from the identity
$$
(1-t x)^{-b_1}(1-ty)^{-b_2}=(1-t x)^{-c}\zav{1-\frac{tx}{tx-1}}^{b_1+b_2-c}\zav{1-\frac{tx}{tx-1}\frac{x-y}{x}}^{-b_2},
$$
we apply $\hypop^{a}_{c}(t)$ on both sides to get:
\begin{align*}
(LHS)&=\hypop^{a}_{c}(t)(1-t x)^{-b_1}(1-ty)^{-b_2}\serovna{F1}F_1\zav{\nadsebou{a}{c};\nadsebou{b_1\quad b_2}{-};tx,ty}.\\
(RHS)&=\hypop^{a}_{c}(t)(1-tx)^{-c}\zav{1-\frac{tx}{tx-1}}^{b_1+b_2-c}\zav{1-\frac{tx}{tx-1}\frac{x-y}{x}}^{-b_2}\\
&\serovna{argscaling}\hypop^{a}_{c}(tx)(1-tx)^{-c}\zav{1-\frac{tx}{tx-1}}^{b_1+b_2-c}\zav{1-\frac{tx}{tx-1}\frac{x-y}{x}}^{-b_2}\\
&\serovna{Pfaffproperty}(1-tx)^{-a}\hypop^{a}_{c}\zav{z}\zav{1-z}^{b_1+b_2-c}\zav{1-z\frac{x-y}{x}}^{-b_2} \qquad \zav{z:=\frac{xt}{xt-1}}\\
&\serovna{F1}(1-xt)^{-a}F_1\zav{\nadsebou{a}{c};\nadsebou{c-b_1-b_2\quad b_2}{-};z,z\frac{x-y}{x}}.
\end{align*} 
Putting $t=1$ we thus obtain:
\begin{equation}\label{F1Pfaff}
F_1\zav{\nadsebou{a}{c};\nadsebou{b_1\quad b_2}{-};x,y}=
(1-x)^{-a}F_1\zav{\nadsebou{a}{c};\nadsebou{c-b_1-b_2\quad b_2}{-};\frac{x}{x-1},\frac{x-y}{x-1}}.
\end{equation} 
\end{example}
\begin{example}
Generally, we can use the identity
$$
\prod_{i=1}^n(1-tx_i)^{-b_i}=(1-tx_1)^{-c}\zav{1-\frac{tx_1}{tx_1-1}}^{b_1+\dots+b_n-c}\prod_{i=2}^n\zav{1-\frac{tx_1}{tx_1-1}\frac{x_1-x_i}{x_1}}^{-b_i},
$$
to obtain
\begin{equation}\label{F1Pfaffgen}
F_1\zav{\nadsebou{a}{c};\nadsebou{{\bf b}}{-};t{\bf x}}=
(1-x)^{-a}F_1\zav{\nadsebou{a}{c};\nadsebou{c-\sum_i b_i\quad b_2\dots b_n}{-};\frac{x_1}{x_1-1},\frac{x_1-x_2}{x-1},\dots,\frac{x_1-x_n}{x_1-1}},
\end{equation} 
where the $F_1$ function is the multivariate generalization of $F_1$ Appell's function  defined by
$$
F_1\zav{\nadsebou{a}{c};\nadsebou{{\bf b}}{-};t{\bf x}}:=\hypop_{c}^{a}(t)(1-tx_1)^{-b_1}\cdots (1-tx_n)^{-b_n},
$$
where ${\bf b},{\bf x}\in\mathbb{R}^n$ such that  ${\bf b}:=(b_1,\dots,b_n)$, ${\bf x}:=(x_1,\dots, x_n)$. 
Notice that $n=1$ corresponds to Gauss's hypergeometric fucntion $\!\! \ _2 F_1$ and $n=2$ corresponds to $F_1$ Appell's function. 
Details are left to the reader.
\end{example}
\begin{example}\label{Pfaffex3}
Perhaps surprisingly, we can also derive \textit{a quadratic} transform for $\!\! \ _2 F_1$. Using
$$
(1-2x)^{-b}=(1-x)^{-2b}\zav{1-\zav{\frac{x}{1-x}}^2}^{-b},
$$
we have
\begin{align*}
\!\! \ _2 F_1\zav{\nadsebou{ a\quad  b}{2b};2x}&\serovna{2F1}\hypop_{2b}^a(1-2x)^{-b}=
\hypop_{2b}^a (1-x)^{-2b}\zav{1-\zav{\frac{x}{1-x}}^2}^{-b}\\
&\serovna{Pfaffproperty}(1-x)^{-a}\hypop_{2b}^a(y)\zav{1-y^2}^{-b}\serovna{secondpower}
(1-x)^{-a}\hypop_{b}^{\frac{a}{2}}\zav{y^2}\hypop_{b+\frac12}^{\frac{a+1}{2}}\zav{y^2}\zav{1-y^2}^{-b}\\
&\serovna{2F1}(1-x)^{-a} \hypop_{b}^{\frac{a}{2}}\zav{y^2} \!\! \ _2 F_1\zav{\nadsebou{b\quad \frac{a+1}{2}}{b+\frac12};y^2}\serovna{2F1}(1-x)^{-a} \hypop_{b}^{\frac{a}{2}}\zav{y^2}\hypop_{b+\frac12}^{b}\zav{y^2}\zav{1-y^2}^{-\frac{a+1}{2}}\\
&\serovna{exchange}(1-x)^{-a} \hypop_{b}^{b}\zav{y^2}\hypop_{b+\frac12}^{\frac{a}{2}}\zav{y^2}\zav{1-y^2}^{-\frac{a+1}{2}}\serovna{2F1}
(1-x)^{-a}\!\! \ _2 F_1\zav{\nadsebou{\frac{a}{2}\quad \frac{a+1}{2}}{b+\frac12};y^2}. 
\end{align*}
Thus we obtained a well known identity:
\begin{equation}\label{2F1q}
\!\! \ _2 F_1\zav{\nadsebou{ a\quad  b}{2b};2x}=(1-x)^{-a}\!\! \ _2 F_1\zav{\nadsebou{\frac{a}{2}\quad \frac{a+1}{2}}{b+\frac12};\zav{\frac{x}{1-x}}^2}.
\end{equation}

\end{example}

\begin{example}\label{Pfaffex4}
A similar elementary identity for the third power, i.e.
$$
(1-zx)^{-b}(1-\bar z x)^{-b}=(1-x)^{-3b}\zav{1+\zav{\frac{x}{1-x}}^3}^{-b},\qquad z+\bar z=3,\ z\bar z=3,
$$
does not gives us a cubic transform of $\!\! \ _2 F_1$ but $F_1$ to $\!\! \ _3 F_2$ reduction, i.e. taking $\hypop_{3b}^a$ of both sides we get:
\begin{equation}\label{F1to3F2}
F_1\zav{\nadsebou{a}{3b};\nadsebou{b\quad b}{-};zx,\bar z x}=(1-x)^{-a}\!\! \ _3 F_2\zav{\nadsebou{ \frac{a}{3}\quad \frac{a+1}{3}\quad \frac{a+2}{3} }{b+\frac13\quad b+\frac23};\zav{\frac{x}{x-1}}^3}.
\end{equation} 
Again, the details are left to the reader.
\end{example}

\begin{example}Once again, we can attempt to generalize this result to multivariate $F_1$ function. From:
$$
\prod_{i=1}^{n-1}(1-(1-z_{i}) x)^{-b}=(1-x)^{-nb}\zav{1-\zav{\frac{x}{x-1}}^n}^{-b},\qquad z_k:=e^{\frac{2\pi\ii k}{n}},
$$
we get the following identity:
\begin{equation}\label{F1gentonFn1}
F_1\zav{\nadsebou{a}{nb};\nadsebou{b\cdots b}{-};(1-z_1)x,\dots,(1-z_{n-1})x}=
(1-x)^{-a}\!\! \ _{n} F_{n-1}\zav{\nadsebou{\frac{a}{n}\dots \frac{a+n-1}{n}}{b+\frac{1}{n}\dots b+\frac{n-1}{n}};\zav{\frac{x}{x-1}}^n}.
\end{equation}
\end{example}
\begin{example}\label{Pfaffex5}
Furthermore, with the aid of the Pfaff property (\ref{Pfaffproperty}) we can establish an alternative representation for $F_1$ function involving only single use of hypergeometrization.
\begin{equation}\label{F1alt}
F_1\zav{\nadsebou{a}{c};\nadsebou{b_1\quad b_2}{-};x,y}=\hypop_{c-b_2}^{b_1}(x)(1-x)^{-a}\!\! \ _2 F_1\zav{\nadsebou{a\quad b_2}{c};\frac{y-x}{1-x}}.
\end{equation}
The argument is as follows:
\begin{align*}
F_1\zav{\nadsebou{a}{c};\nadsebou{c-b_2\quad b_2}{-};x,y}&\stackrel{(\ref{F1Pfaff})}{=}(1-x)^{-a} F_1\zav{\nadsebou{a}{c};\nadsebou{0\quad b_2}{-};\frac{x}{x-1},\frac{y-x}{1-x}}=(1-x)^{-a}\!\! \ _2 F_1\zav{\nadsebou{a\quad b_2}{c};\frac{y-x}{1-x}}.
\end{align*}
Now just apply $\hypop_{b_1}^{c-b_2}(x)$ to both sides.

This representation allows us, for instance, to easily see that the following identity holds: 
\begin{equation}\label{F1(x,1)}
F_1\zav{\nadsebou{a}{c};\nadsebou{b_1\quad b_2}{-};x,1}=\frac{\Gamma(c)\Gamma(c-a-b_2)}{\Gamma(c-a)\Gamma(c-b_2)}\!\! \ _2 F_1\zav{\nadsebou{a\quad b_1}{c-b_2};x},\qquad \Re(c-a-b_2)>0.
\end{equation}

Just put $y=1$ and use the well known Gauss's summation formula (see \cite[15.4.20]{dlmf15.4})!
\begin{equation}\label{2F1(1)}
\!\! \ _2 F_1\zav{\nadsebou{a\quad b}{c};1}=\frac{\Gamma(c)\Gamma(c-a-b)}{\Gamma(c-a)\Gamma(c-b)},\qquad \Re(c-a-b)>0.
\end{equation}
\end{example}

\begin{example}\label{Pfaffex6}
We can also obtain some transform for $F_2$ function. Take the following identity 
$$
(1-x-y)^{-a}=(1-x)^{-a}(1-y)^{-a}\zav{1-\frac{xy}{(1-x)(1-y)}}^{-a},
$$
and apply operators $\hypop_{a}^{b_1}(x)\hypop_{a}^{b_2}(y)$ on both sides.
\begin{align*}
(LHS)&=\hypop_{a}^{b_1}(x)\hypop_{a}^{b_2}(y)(1-x-y)^{-a}\serovna{F2}F_2\zav{\nadsebou{a}{-};\nadsebou{b_1}{a}\nadsebou{b_1}{a};x,y}.\\
(RHS)&=\hypop_{a}^{b_1}(x)\hypop_{a}^{b_2}(y)(1-x)^{-a}(1-y)^{-a}\zav{1-\frac{xy}{(1-x)(1-y)}}^{-a}\\
&\serovna{linearity}\hypop_{a}^{b_1}(x)(1-x)^{-a}\hypop_{a}^{b_2}(y)(1-y)^{-a}\zav{1-\frac{xy}{(1-x)(1-y)}}^{-a}\\
&\serovna{Pfaffproperty}(1-x)^{-b_1}(1-y)^{-b_2}\hypop_{a}^{b_1}\zav{\tilde x}\hypop_{a}^{b_2}(\tilde y)\zav{1-\tilde x\tilde y}^{-a},\qquad \zav{\tilde x:=\frac{x}{x-1},\ \tilde y:=\frac{y}{y-1}}\\
&\serovna{2F1}(1-x)^{-b_1}(1-y)^{-b_2}\!\! \ _2 F_1\zav{\nadsebou{b_1\quad b_2}{a};\tilde x\tilde y}.
\end{align*}
Altogether we have
\begin{equation}\label{F2to2F1}
F_2\zav{\nadsebou{a}{-};\nadsebou{b_1}{a}\nadsebou{b_1}{a};x,y}=(1-x)^{-b_1}(1-y)^{-b_2}\!\! \ _2 F_1\zav{\nadsebou{b_1\quad b_2}{a};\frac{xy}{(x-1)(y-1)}}.
\end{equation}
\end{example}
\begin{example}
We will now compute the following link between $G_2$ and $F_2$ functions:
\begin{equation}\label{G2toF2}
G_2\zav{a\quad c; \nadsebou{b_1\quad b_2}{-};x,y}=(1+x)^{-b_1}(1+y)^{-b_2}F_2\zav{\nadsebou{1-c-a}{-};\nadsebou{b_1}{1-c}\nadsebou{b_2}{1-a};\frac{x}{x+1},\frac{y}{y+1}}.
\end{equation}
Once again, there is an elementary identity in behind the transform:
\begin{equation}\label{G2toF2el}
(1+y)^{-c}(1+x)^{-a}(1-xy)^{c+a-1}
=(1+y)^{a-1}(1+x)^{c-1}\zav{1-\frac{y}{y+1}-\frac{x}{x+1}}^{c+a-1}.
\end{equation}
To prove (\ref{G2toF2}) simply apply $\hypop_{1-c}^{b_1}(x)\hypop_{b_2}^{1-a}(y)$ on both sides of (\ref{G2toF2el}) and use the Pfaff property when appropriate.  
\end{example}
\section{Euler property}\label{S8}
Remember that Euler transform (\ref{Euler}) of $\!\! \ _2 F_1$ function can be obtained by applying the Pfaff transform (\ref{Pfaff}) twice (with a swapping of parameters). The same procedure can be also applied on the level of hypergeoemtrization:
\begin{proposition}\label{Eulerpropertyprop}
Let $a,\ b, \ c\in \C$, such that $1-c\not \in \N$. Then on functions analytic near origin it holds:
\begin{equation}\label{Eulerproperty}
(1-x)^{a+b-c}\hypop_{c}^a(1-x)^{-b}=\hypop_{c}^{c-b}(1-x)^{-(c-a)}\hypop_{c-b}^{a}.
\end{equation}
\end{proposition}
\begin{proof}
\begin{align*}
\hypop_{c}^a(x)&\stackrel{(\ref{Pfaffproperty})}{=}(1-y)^{a}\hypop_{c}^{a}(y)(1-y)^{-c}, & y&:=\frac{x}{x-1},\\
&\stackrel{(\ref{exchange})}{=}(1-x)^{-a}\hypop_{c}^{b}(y)\hypop_{b}^{a}(y)(1-x)^{c}\\
&\stackrel{(\ref{Pfaffproperty})}{=}
(1-x)^{a+b}\hypop_{c}^{b}(x)(1-x)^{-c+a}\hypop_{b}^{a}(x)(1-x)^{c-b}.
\end{align*}
This is what we want just in different form.
\end{proof}
 
\begin{example}
Applying (\ref{Eulerproperty}) on the constant function $1$ we get
\begin{align*}
LHS&=(1-x)^{a+b-c}\hypop_{c}^a(1-x)^{-b}1=(1-x)^{a+b-c}\!\! \ _2 F_1\zav{\nadsebou{a\quad b}{c};x}.\\
RHS&=\hypop_{c}^{c-b}(1-x)^{-(c-a)}\hypop_{c-b}^{a}1=\hypop_{c}^{c-b}(1-x)^{-(c-a)}=\!\! \ _2 F_1\zav{\nadsebou{c-b\quad c-a }{c};x},
\end{align*}
which is exactly Euler transform (\ref{Euler}).
\end{example}
\begin{example}
We can also derive Euler-like transform for $\!\! \ _3 F_2$ function in the form
\begin{equation}\label{3F2Euler}
\!\! \ _3 F_2\zav{\nadsebou{a_1\quad a_2\quad a_3}{c_1\quad c_2};x}=(1-x)^{\sigma}\hypop_{c_1}^{\sigma+a_1}(1-x)^{-(c_1-a_1)}\!\!\ _3F_2\zav{\nadsebou{a_1\quad c_2-a_2\quad c_2-a_3}{\sigma+a_1\quad c_2};x},
\end{equation}
where the so-called \textit{parameter excess} $\sigma$ is $\sigma:=c_1+c_2-a_1-a_2-a_3$.
Proof is done by the following argument:
\begin{align*}
\!\! \ _3 F_2\zav{\nadsebou{a_1\quad a_2\quad a_3}{c_1\quad c_2};x}&=\hypop_{c_1}^{a_1}
\!\! \ _2 F_1\zav{\nadsebou{a_2\quad a_3}{ c_2};x}\\
&\serovna{Euler}\hypop_{c_1}^{a_1}(1-x)^{c_2-a_2-a_3}
\!\! \ _2 F_1\zav{\nadsebou{c_2-a_2\quad c_2-a_3}{ c_2};x}\\
&\serovna{Eulerproperty}(1-x)^{\sigma}\hypop_{c_1}^{\sigma+a_1}(1-x)^{-(c_1-a_1)}\hypop_{\sigma+a_1}^{a_1}\!\! \ _2 F_1\zav{\nadsebou{c_2-a_2\quad c_2-a_3}{c_2};x}\\
&=(1-x)^{\sigma}\hypop_{c_1}^{\sigma+a_1}(1-x)^{-(c_1-a_1)}\!\! \ _3 F_2\zav{\nadsebou{a_1\quad c_2-a_2\quad c_2-a_3}{\sigma+a_1\quad c_2};x}.
\end{align*}
\end{example}
An important corollary that will be useful later on is the following:
\begin{corollary} Let $\szav{c_j}_{j\in\mathbb{Z}}$, $\szav{a_j}_{j\in\mathbb{Z}}$ are given sequences of complex numbers. Then for any $n\in \mathbb{Z}$ it holds:
\begin{equation}\label{Eulergen}
\prod_{j=1}^n (1-x)^{c_j}\hypop_{a_{j-1}}^{a_j}(x)=(1-x)^{c_1-\tilde c_1}
\zav{\prod_{j=1}^n (1-x)^{\tilde c_j}\hypop_{\tilde a_{j-1}}^{\tilde a_j}(x)} (1-x)^{a_{n-1}-\tilde a_{n-1}}\hypop_{\tilde a_n}^{a_n}(x),
\end{equation}
and
\begin{equation}\label{Eulergenrev}
\prod_{j=1}^n (1-x)^{\tilde c_j}\hypop_{\tilde a_{j-1}}^{\tilde a_j}(x)=(1-x)^{\tilde c_1-c_1}
\zav{\prod_{j=1}^n (1-x)^{c_j}\hypop_{a_{j-1}}^{a_j}(x)}\hypop_{ a_n}^{\tilde a_n}(x)  (1-x)^{\tilde a_{n-1}- a_{n-1}},
\end{equation}
where
$$
\tilde a_j:=a_0+\sum_{k=1}^j c_k,\qquad \tilde c_j:=c_j+c_{j-1}-a_{j-1}+a_{j-2}.
$$
\end{corollary}
\begin{remark}
We claim  that equations (\ref{Eulergen}), (\ref{Eulergenrev}) are valid even for negative $n$. In that case, concerned products must be interpreted as in (\ref{negativeproduct}) and, in the same way, we define
\begin{equation}\label{negativesum}
\tilde a_0:=a_0,\qquad \tilde a_{-j}:=a_0-\sum_{k=1}^{j}c_{1-k},\qquad j\in\mathbb{N}.
\end{equation}
\end{remark}
\begin{proof}
We are going to prove (\ref{Eulergen}) only.  The second identity (\ref{Eulergenrev}) is just its inverse.
There are two cases to consider. 

\textit{Case 1.} Suppose $n\geq 0$.
Using the obvious identity
\begin{equation}\label{preusporadat}
\prod_{j=1}^n A_j B_j=A_1\zav{\prod_{j=2}^n B_{j-1}A_j}B_n,
\end{equation}
which holds for any sequences of linear operators $A_j,B_j$ (and in fact for any integer $n$) we can see that
\begin{align*}
\prod_{j=1}^n (1-x)^{c_j}\hypop_{a_{j-1}}^{a_j}(x)&\stackrel{(\ref{exchange})}{=}\prod_{j=1}^n (1-x)^{c_j}\hypop_{a_{j-1}}^{\tilde a_j}(x)\hypop_{\tilde a_j}^{a_j}(x)\stackrel{(\ref{preusporadat})}{=}(1-x)^{c_1}\hypop_{a_0}^{\tilde a_1}(x)\zav{\prod_{j=2}^n \hypop_{\tilde a_{j-1}}^{a_{j-1}}(x)(1-x)^{c_j}\hypop_{a_{j-1}}^{\tilde a_j}(x) }\hypop_{\tilde a_n}^{a_n}(x).\\
\intertext{Note that $\tilde a_j-\tilde a_{j-1}=c_j$. Therefore we can use Euler property to obtain:}
\prod_{j=1}^n (1-x)^{c_j}\hypop_{a_{j-1}}^{a_j}(x)&\stackrel{(\ref{Eulerproperty})}{=}
(1-x)^{c_1}\hypop_{a_0}^{\tilde a_1}(x)\zav{\prod_{j=2}^n (1-x)^{\tilde a_j-a_{j-1}}\hypop_{\tilde a_{j-1}}^{\tilde a_{j}}(x)(1-x)^{a_{j-1}-\tilde a_{j-1}} }\hypop_{\tilde a_n}^{a_n}(x)\\
&\stackrel{(\ref{preusporadat})}{=}
(1-x)^{c_1}\hypop_{a_0}^{\tilde a_1}(x)(1-x)^{\tilde a_2-a_1}\hypop_{\tilde a_1}^{\tilde a_2}(x)\zav{\prod_{j=3}^n (1-x)^{\tilde c_j}\hypop_{\tilde a_{j-1}}^{\tilde a_{j}}(x) }(1-x)^{a_{n-1}-\tilde a_{n-1}}\hypop_{\tilde a_n}^{a_n}(x),\\
\intertext{here we have used the fact that $\tilde c_j=\tilde a_j-a_{j-1}+a_{j-2}-\tilde a_{j-2}$ since $\tilde a_{j}-\tilde a_{j-2}=c_j+c_{j-1}$. Observe also that $\tilde c_2=\tilde a_2-a_1$ and $\tilde a_0=a_0$. Thus}
&=(1-x)^{c_1-\tilde c_1}\zav{\prod_{j=1}^n (1-x)^{\tilde c_j}\hypop_{\tilde a_{j-1}}^{\tilde a_{j}}(x) }(1-x)^{a_{n-1}-\tilde a_{n-1}}\hypop_{\tilde a_n}^{a_n}(x).
\end{align*}
This proves (\ref{Eulergen}) for $n\geq 0$. 

\textit{Case 2.} The case $n<0$ we will prove by induction. Renaming $n=-n$ and using the definition for ``negative'' product (\ref{negativeproduct}) we have to show that
$$
\prod_{j=1}^n \hypop_{a_{1-j}}^{a_{-j}}(x) (1-x)^{-c_{1-j}}=(1-x)^{c_1-\tilde c_1}
\zav{\prod_{j=1}^n \hypop_{\tilde a_{1-j}}^{\tilde a_{-j}}(x)(1-x)^{-\tilde c_{1-j}}} (1-x)^{a_{-n-1}-\tilde a_{-n-1}}\hypop_{\tilde a_{-n}}^{a_{-n}}(x),
$$
for all $n=0,1,2,\dots$. The base case $n=0$ is trivial.

For the induction steps  
\begin{align*}
&\prod_{j=1}^{n+1} \hypop_{a_{1-j}}^{a_{-j}}(x) (1-x)^{-c_{1-j}}=
\zav{\prod_{j=1}^{n} \hypop_{a_{1-j}}^{a_{-j}}(x) (1-x)^{-c_{1-j}}}\hypop_{a_{-n}}^{a_{-n-1}}(x) (1-x)^{-c_{-n}}\\
&=(1-x)^{c_1-\tilde c_1}
\zav{\prod_{j=1}^n \hypop_{\tilde a_{1-j}}^{\tilde a_{-j}}(x)(1-x)^{-\tilde c_{1-j}}} (1-x)^{a_{-n-1}-\tilde a_{-n-1}}\hypop_{\tilde a_{-n}}^{a_{-n}}(x)
\hypop_{a_{-n}}^{a_{-n-1}}(x) (1-x)^{-c_{-n}}\\
&\stackrel{(\ref{exchange})}{=}
(1-x)^{c_1-\tilde c_1}
\zav{\prod_{j=1}^n \hypop_{\tilde a_{1-j}}^{\tilde a_{-j}}(x)(1-x)^{-\tilde c_{1-j}}} (1-x)^{a_{-n-1}-\tilde a_{-n-1}}
\hypop_{\tilde a_{-n}}^{a_{-n-1}}(x) (1-x)^{-c_{-n}}\\
&\stackrel{(\ref{Eulerproperty})}{=}
(1-x)^{c_1-\tilde c_1}
\zav{\prod_{j=1}^n \hypop_{\tilde a_{1-j}}^{\tilde a_{-j}}(x)(1-x)^{-\tilde c_{1-j}}} 
\hypop_{\tilde a_{-n}}^{\tilde a_{-n-1}}(x)(1-x)^{a_{-n-1}-\tilde a_{-n}} \hypop_{\tilde a_{-n-1}}^{a_{-n-1}}(x)\\
&=
(1-x)^{c_1-\tilde c_1}
\zav{\prod_{j=1}^{n+1} \hypop_{\tilde a_{1-j}}^{\tilde a_{-j}}(x)(1-x)^{-\tilde c_{1-j}}} 
(1-x)^{-\tilde c_{-n}+a_{-n-1}-\tilde a_{-n}} \hypop_{\tilde a_{-n-1}}^{a_{-n-1}}(x)\\
&=
(1-x)^{c_1-\tilde c_1}
\zav{\prod_{j=1}^{n+1} \hypop_{\tilde a_{1-j}}^{\tilde a_{-j}}(x)(1-x)^{-\tilde c_{1-j}}} 
(1-x)^{a_{-n-2}-\tilde a_{-n-2}} \hypop_{\tilde a_{-n-1}}^{a_{-n-1}}(x),
\end{align*}
where the last equality stems form the definition of $\tilde c_{-n}$ and $\tilde a_{-n-2}$. Which is what we want. 
Thus we have proven (\ref{Eulergen}) for all integer $n$.

\end{proof}
\begin{example}
For $c_j=a_j-a_{j-1}$ it holds
$$
\tilde c_j=c_j,\qquad \tilde a_j=a_j,
$$
and equality (\ref{Eulergen}) is a simple identity.
\end{example}
\begin{example}
If $c_j=a_j-a_{j-1}+\alpha$ for some fixed $\alpha\in \mathbb{C}$ we have
$$
\tilde c_j=c_j+\alpha,\qquad \tilde a_j=a_j+\alpha j.
$$
Notice that 
$\tilde c_j=\tilde a_j-\tilde a_{j-1}+\alpha$. We can therefore repeat the process. If we do it $m$ times we obtain the following identity:
\begin{equation}\label{Eulergenm}
\prod_{j=1}^{n}(1-x)^{c_j}\hypop_{a_{j-1}}^{a_j}=(1-x)^{-\alpha m}
\zav{\prod_{j=1}^{n}(1-x)^{c_j+m\alpha}\hypop_{a_{j-1}+m\alpha(j-1)}^{a_j+m\alpha j}}\zav{\prod_{k=1}^m (1-x)^{-\alpha(n-1)}\hypop_{a_n+(m+1-k)\alpha n}^{a_n+(m-k)\alpha n}}.
\end{equation}
Now, if we solve for the first product on the right by multiplying by  the inverse of the second product from the right and by the factor $(1-x)^{\alpha m}$ from the left and then rename the sequences $c_j\to c_j-m\alpha$ and $a_j\to a_j-m\alpha j$, we obtain an inverse expression which reads:
\begin{equation}\label{Eulergen-m}
\prod_{j=1}^{n}(1-x)^{c_j}\hypop_{a_{j-1}}^{a_j}=(1-x)^{\alpha m}
\zav{\prod_{j=1}^{n}(1-x)^{c_j-m\alpha}\hypop_{a_{j-1}-m\alpha(j-1)}^{a_j-m\alpha j}}\zav{\prod_{k=1}^{m} \hypop_{a_n+(k-1-m)\alpha n}^{a_n+(k-m)\alpha n}(1-x)^{\alpha(n-1)}}.
\end{equation}
But observe this is exactly the same formula which we would get if we put $m=-m$ into (\ref{Eulergenm}) and interpret the product as usual (see (\ref{negativeproduct})). 

Therefore the formula (\ref{Eulergenm}) is in fact true for all integers $m\in\mathbb{Z}$.
\end{example}


%
%
\section{Change of coordinates}\label{S10}
The Pfaff property (\ref{Pfaffproperty}) along with scaling of the argument (\ref{argscaling}) and argument's power law (\ref{nthpower}), i.e. the following list:
\begin{align*}
&(\ref{argscaling})\qquad\hypop_c^a (x)=\hypop_c^a (y), &  y&=\alpha x.\\
&(\ref{secondpower})\qquad\hypop_c^a (x)=\hypop_{\frac{c}{2}}^{\frac{a}{2}}(y)\hypop_{\frac{c+1}{2}}^{\frac{a+1}{2}}(y), &  y&=x^2.\\
&(\ref{nthpower})\qquad\hypop_c^a (x)=\hypop_{\frac{c}{n}}^{\frac{a}{n}}(y)\hypop_{\frac{c+1}{n}}^{\frac{a+1}{n}}(y)\cdots \hypop_{\frac{c+n-1}{n}}^{\frac{a+n-1}{n}}(y) , &  y&=x^n.\\
&(\ref{Pfaffproperty})\qquad\hypop^a_c\zav{x}=(1-y)^{a}\hypop^a_c\zav{y}(1-y)^{-c}, & y&=\frac{x}{x-1}.
\end{align*}
can be viewed as an instances of \textit{change of variable} $x\to y$. Are there any more? Obviously, we can produce additional identities just by \textit{combining} (\ref{Pfaffproperty}), (\ref{argscaling}) and (\ref{nthpower}), for example:
\begin{align}
&(1-x)^{\frac{a}{2}}\hypop^a_c\zav{x}(1-x)^{-\frac{c}{2}}=\hypop^{\frac{a}{2}}_{\frac{c+1}{2}}\zav{y}(1-y)^{-\frac{c-a}{2}}\hypop_{\frac{c}{2}}^{\frac{a+1}{2}}(y), & y&=\frac{x^2}{4(x-1)}.\label{Qt5}\\
&\zav{1-\frac{x}{2}}^{a}\hypop^a_c\zav{x}\zav{1-\frac{x}{2}}^{-c}=\hypop^{\frac{a}{2}}_{\frac{c+1}{2}}\zav{y}\hypop_{\frac{c}{2}}^{\frac{a+1}{2}}(y), & y&=\frac{x^2}{(2-x)^2}.\label{Qt6}\\
&(1-x^2)^{\frac{a+1}{2}}\hypop_{c}^a(x) (1-x^2)^{-\frac{c+1}{2}}=\hypop_{\frac{c}{2}}^{\frac{a+1}{2}}(y)(1-y)^{-\frac{c-a}{2}}\hypop_{\frac{c+1}{2}}^{\frac{a}{2}}(y), & y&=\frac{x^2}{x^2-1}.\label{Qt7}
\end{align}
For the proof, define the following functions:
\begin{align*}
S_\alpha(x)&=\alpha x,& \text{Scaling.}\\
M_\alpha(x)&=x^\alpha, & \text{Power.}\\
P(x)&=\frac{x}{x-1}, & \text{Pfaff.}
\end{align*}
Their properties are:
\begin{align*}
S_\alpha\circ S_\beta&=S_{\alpha\beta}, & S_1&=Id,\\
M_\alpha\circ M_\beta&=M_{\alpha\beta}, & M_1&=Id,\\
P\circ P&=Id.
\end{align*}
We have
\begin{align*}
\frac{x^2}{x^2-1}&=P\circ M_2(x), & \frac{x^2}{4(x-1)}&=P\circ M_2\circ P\circ S_{\frac12}(x),\\
\zav{\frac{x}{2-x}}^2&=M_2\circ P\circ S_\frac12(x).
\end{align*}
Thus the identities (\ref{Qt5}), (\ref{Qt6}), (\ref{Qt7}) are direct consequences of already established properties (\ref{Pfaffproperty}), (\ref{argscaling}), (\ref{nthpower}).
\begin{example}
Applying the identity (\ref{Qt5}) on the constant function 1 we get:
\begin{equation}\label{2F1Q2}
\zav{1-x}^{\frac{a}{2}}\!\! \ _2 F_1\zav{\nadsebou{a\quad \frac{c}{2}}{c};x}=\!\! \ _2 F_1\zav{\nadsebou{\frac{a}{2}\quad \frac{c-a}{2}}{\frac{c+1}{2}};\frac{x^2}{4(x-1)}},
\end{equation}
a quadratic transform for $\!\! \ _2 F_1$ (the identity 15.8.14 in \cite{dlmf15.8}).
\end{example}
Evidently, \textit{any} composition chain of $P,\ S_\alpha\ M_\alpha$ functions will lead to a valid change of coordinates. For instance:
\begin{align*}
\frac{x^2}{ax+b}&=S_{-\frac{4b}{a^2}}\circ P\circ M_2\circ P\circ S_{-\frac{a}{2b}}(x).
\end{align*}
A function that cannot be obtain by any finite combination of $P,\ S_\alpha,\ M_\alpha$ is
$$
Q(x):=\frac{-4x}{(1-x)^2},
$$
but the corresponding change of variable is the following:
\begin{proposition}\label{Qprop}
Let $\beta:=\frac{a+c-1}{2}$. Then it holds:
\begin{equation}\label{Qt1monom}
(1-x)^{2\beta}\hypop^a_c\zav{x}(1-x)^{-2\beta}=\hypop^{\beta}_c\zav{y}(1-y)^{-\frac{c-a}{2}}\hypop_{\beta}^{a}(y),\qquad  y:=\frac{-4x}{(1-x)^2}.
\end{equation}
\end{proposition}
\begin{proof} As always, it is sufficient to prove the formula (\ref{Qt1monom}) only on powers of $y$.
  The proof is based on a ``quadratic transform'' of $\!\! \ _2 F_1$ function valid for $\abs{x}<1$:
\begin{equation}\label{2F1Q1}
\!\! \ _2 F_1\zav{\nadsebou{a\quad b}{a-b+1};x}=(1-x)^{-a}\!\! \ _2 F_1\zav{\nadsebou{\frac{a}{2}\quad \frac{a}{2}-b+\frac12}{a-b+1};\frac{-4x}{(1-x)^2}},\qquad \abs{x}<1. 
\end{equation}
See \cite[15.8.16]{dlmf15.8}.
Let $1+\alpha\in\mathbb{N}$. Then we have
\begin{align*}
LHS&=(1-x)^{2\beta}\hypop^a_c\zav{x}(1-x)^{-2\beta}y^\alpha=(1-x)^{2\beta}\hypop^a_c\zav{x}(1-x)^{-2\beta}(-4x)^\alpha (1-x)^{-2\alpha}\\
&\serovna{shift}(1-x)^{2\beta}(-4x)^\alpha\frac{(a)_\alpha}{(c)_\alpha}\hypop^{a+\alpha}_{c+\alpha}\zav{x}(1-x)^{-2(\beta+\alpha)}\\
&\serovna{2F1}(1-x)^{2\beta}(-4x)^\alpha\frac{(a)_\alpha}{(c)_\alpha}\!\! \ _2 F_1\zav{\nadsebou{a+\alpha\quad 2(\beta+\alpha)}{c+\alpha};x}\\
&\serovna{2F1Q1}(1-x)^{2\beta}(-4x)^\alpha\frac{(a)_\alpha}{(c)_\alpha}(1-x)^{-2\beta-2\alpha}\!\! \ _2 F_1\zav{\nadsebou{\beta+\alpha\quad \beta-a+\frac12}{c+\alpha};y}\\
&\serovna{2F1}y^\alpha\frac{(a)_\alpha}{(c)_\alpha}\hypop_{c+\alpha}^{\beta+\alpha}(y)(1-y)^{a-\beta-\frac12}\serovna{shift}\hypop_{c}^{\beta}(y) \frac{(a)_\alpha}{(\beta)_\alpha}y^\alpha(1-y)^{-\frac{c-a}{2}}=\hypop_{c}^{\beta}(y)(1-y)^{-\frac{c-a}{2}}\hypop_{\beta}^a y^\alpha\\
&=RHS.
\end{align*}
\end{proof}
\begin{example} 
Using (\ref{Qt1monom}) on a constant function $1$ we obtain
$$
(1-x)^{c+a-1}\!\! \ _2 F_1\zav{\nadsebou{c+a-1\quad a}{c};x}=\!\! \ _2F_1\zav{\nadsebou{\frac{c-a}{2}\quad \frac{a+c-1}{2}}{c};-\frac{4x}{(1-x)^2}},
$$
thus we recovered (\ref{2F1Q1}). (See \cite[15.8.6.]{dlmf15.8})

Now, shift the parameters by $a\to a+b $, $c\to c-b$ so we have
$$
(1-x)^{c+a-1}\!\! \ _2 F_1\zav{\nadsebou{c+a-1\quad a+b}{c-b};x}=\!\! \ _2F_1\zav{\nadsebou{\frac{c-a}{2}-b\quad \frac{a+c-1}{2}}{c-b};y},
$$
and apply transform again to get
\begin{align*}
(1-x)^{c+a-1}\!\! \ _3 F_2\zav{\nadsebou{c+a-1\quad a+b\quad a}{c-b\quad c};x}&=
\hypop_{\frac{a+c-1}{2}}^{c}(y)(1-y)^{-\frac{c-a}{2}} \!\! \ _2F_1\zav{\nadsebou{\frac{c-a}{2}-b\quad a}{c-b};y}\\
&=\hypop_{\frac{a+c-1}{2}}^{c}(y)\!\! \ _2F_1\zav{\nadsebou{\frac{c+a}{2}\quad c-a-b}{c-b};y}\\
&=\!\! \ _3F_2\zav{\nadsebou{\frac{c-a}{2}-b\quad \frac{c+a}{2}\quad \frac{a+c-1}{2}}{c-b\quad c};-\frac{4x}{(1-x)^2}},
\end{align*}
a quadratic formula for $\!\! \ _3 F_2$! (See \cite[16.6.1.]{dlmf16.6}.)
\end{example}
\begin{example}
Consider the function 
$$
g(x):=(1-yt)^{-\beta},\qquad y:=\frac{-4x}{(1-x)^2},\qquad \beta:=\frac{a+c-1}{2}.
$$
Note that 
$$
1-yt=\frac{1-2x(1-2t)+x^2}{(1-x)^2}=\frac{(1-\tau_+ x)(1-\tau_- x)}{(1-x)^2},
$$
where $\tau_\pm$ are complex numbers such that $\tau_++\tau_-=2-4t$, $\tau_+\tau_-=1$, i.e.
$$
\tau_{\pm}:=2\zav{(2t-1)^2\pm\sqrt{t(t-1)}}.
$$
 Thus
$$
g(x)=(1-\tau_+ x)^{-\beta}(1-\tau_- x)^{-\beta}(1-x)^{2\beta}.
$$
Applying (\ref{Qt1monom}) on the function $g$ we obtain: 
\begin{align*}
RHS&=\hypop^{\beta}_c\zav{y}(1-y)^{-\frac{c-a}{2}}\hypop_{\beta}^{a}(y)(1-yt)^{-\beta}\\
&\serovna{2F1}\hypop^{\beta}_c\zav{y}(1-y)^{-\frac{c-a}{2}}(1-yt)^{-a}\serovna{F1}
F_1\zav{\nadsebou{\beta}{c};\nadsebou{\frac{c-a}{2}\quad a }{-};y,yt}.\\
LHS&=(1-x)^{2\beta}\hypop^a_c\zav{x}(1-x)^{-2\beta}g(x)=(1-x)^{2\beta}\hypop^a_c\zav{x}(1-\tau_+ x)^{-\beta}(1-\tau_- x)^{-\beta}\\
&\serovna{F1}(1-x)^{2\beta}F_1\zav{\nadsebou{a}{c};\nadsebou{\beta\quad \beta}{-};\tau_+ x,\tau_- x}.
\end{align*}
Altogether we discover a \textit{quadratic} transform for $F_1$:
\begin{equation}\label{F1Q}
F_1\zav{\nadsebou{a}{c};\nadsebou{\frac{a+c-1}{2}\quad\frac{a+c-1}{2} }{-};\tau_+ x,\tau_- x}=(1-x)^{1-a-c}F_1\zav{\nadsebou{\frac{a+c-1}{2}}{c};\nadsebou{\frac{c-a}{2}\quad a}{-};\frac{-4x}{(1+x)^2},\frac{-4x t}{(1+x)^2}},
\end{equation}
where
$$
\tau_{\pm}:=2\zav{(2t-1)^2\pm\sqrt{t(t-1)}}.
$$

For more quadratic transforms of Appell's function see \cite{Carlson2}.
\end{example}
We can of course consider also combinations of $Q$ with other functions:
\begin{proposition}\label{ChangeofvarP} Let $\beta:=\frac{a+c-1}{2}$. For generic values of $a,c\in\C$ it holds:
\begin{align}
&(1+x)^{2\beta}\hypop^a_c\zav{x}(1+x)^{-2\beta}=\hypop^{\beta}_c\zav{y}(1-y)^{-\frac{c-a}{2}}\hypop_{\beta}^{a}(y), & y&=\frac{4x}{(1+x)^2}.\label{Qt2}\\
&(1-x)^{1-c}\hypop^a_c\zav{x}(1-x)^{a-1}=\hypop^{\beta}_c\zav{y}(1-y)^{-\frac{c-a}{2}}\hypop_{\beta}^{a}(y), & y&=4x(1-x).\label{Qt3}\\
&(1-x)^{1-c}\hypop^a_c\zav{x}(1-x)^{a-1}=(1-y)^{\beta}\hypop^{\beta}_c\zav{y}(1-y)^{-\frac{c-a}{2}}\hypop_{\beta}^{a}(y)(1-y)^{-\beta}, & y&=\frac{4x(x-1)}{(1-2x)^2}.\label{Qt4}
\end{align}
\end{proposition}
\begin{proof}
These identities can be obtained, considering the following compositions: 
\begin{align*}
\frac{4x}{(1+x)^2}&=P\circ Q(x)=Q\circ S_{-1}(x), & 4x(1-x)&=Q\circ P(x)\\
\frac{4x(x-1)}{(1-2x)^2}&=P\circ Q\circ P(x). 
\end{align*}
\end{proof}
\begin{example}
Consider the following elementary identity:
\begin{equation}\label{elidex}
(1-x)^{-3\alpha}\zav{1-\zav{\frac{x}{x-1}}^3}^{-\alpha}=\zav{1-3x(1-x)}^{-\alpha},\qquad \alpha\in\mathbb{C}.
\end{equation}
Applying the operator 
$$
(1-x)^{a-3\alpha}\hypop_{3\alpha-a+1}^a(1-x)^{a-1},
$$
on the LHS of (\ref{elidex}) we get:
\begin{align*}
\text{LHS of (\ref{elidex})}&\to(1-x)^{a-3\alpha}\hypop_{3\alpha-a+1}^a (1-x)^{a-1-3\alpha}\zav{1-\zav{\frac{x}{x-1}}^3}^{-\alpha}\\
&\stackrel{(\ref{Pfaffproperty})}{=}(1-x)^{-3\alpha}\!\! \ _4 F_3\zav{\nadsebou{\frac{a}{3}\quad \frac{a+1}{3}\quad  \frac{a+2}{3}\quad \alpha }{\alpha+\frac{1-a}{3}\quad\alpha+\frac{2-a}{3}\quad \alpha+\frac{3-a}{3} };\zav{\frac{x}{x-1}}^3}.\\
\intertext{Applying the same operator also on the RHS of (\ref{elidex}) yields:}
\text{RHS of (\ref{elidex})}&\to(1-x)^{a-3\alpha}\hypop_{3\alpha-a+1}^a (1-x)^{a-1}\zav{1-3x(1-x)}^{-\alpha}\\
&\stackrel{(\ref{Qt3})}{=}\hypop_{3\alpha-a+1}^{\frac32 \alpha}(y)(1-y)^{-\frac{3\alpha-2a+1}{2}}\hypop_{\frac32 \alpha}^{a}(y)\zav{1-\frac34 y}^{-\alpha}\qquad (y=4x(1-x))\\
&=\hypop_{3\alpha-a+1}^{\frac32 \alpha}(y) (1-y)^{-\frac{3\alpha-2a+1}{2}}\!\! \ _2F_1\zav{\nadsebou{a\quad \alpha}{\frac32 \alpha};\frac34 y}.
\end{align*}
Thus
\begin{equation}\label{semicubic4F3}
(1-x)^{-3\alpha}\!\! \ _4 F_3\zav{\nadsebou{\frac{a}{3}\quad \frac{a+1}{3}\quad  \frac{a+2}{3}\quad \alpha }{\alpha+\frac{1-a}{3}\quad\alpha+\frac{2-a}{3}\quad \alpha+\frac{3-a}{3} };\zav{\frac{x}{x-1}}^3}
\end{equation}
$$
=\hypop_{3\alpha-a+1}^{\frac32 \alpha}(y) (1-y)^{-\frac{3\alpha-2a+1}{2}}\!\! \ _2F_1\zav{\nadsebou{a\quad \alpha}{\frac32 \alpha};\frac34 y},\qquad y:=4x(1-x).
$$
\end{example}
\begin{example}
Putting $3\alpha=2a$ in (\ref{semicubic4F3}) and using (\ref{F1}) we obtain a semi-cubic transform for $F_1$ function!
\begin{equation}\label{F1semicubic}
(1-x)^{-2a}\!\! \ _2 F_1\zav{\nadsebou{\frac{a}{3}\quad \frac{2a}{3}}{\frac{a}{3}+1};\zav{\frac{x}{x+1}}^3}=F_1\zav{\nadsebou{a}{a+1};\nadsebou{\frac12\quad \frac23 a}{-};4x(1-x),3x(1-x)}.
\end{equation}
\end{example}
\begin{example}
Putting $x=\frac12$ into (\ref{F1semicubic}) we get the following summation formula for $\!\! \ _2 F_1(3/4)$:
\begin{equation}\label{2F1(3/4)}
\!\! \ _2 F_1\zav{\nadsebou{a\quad \frac23 a}{a+\frac12};\frac34}=\frac{4^{\frac23} \Gamma\zav{1+\frac13 a}\Gamma\zav{a+\frac12}}{\Gamma\zav{\frac12+\frac13 a}\Gamma\zav{1+a}}.
\end{equation} 
This follows from the identity (\ref{F1(x,1)}) and a well known summation formula
$$
\!\! \ _2 F_1\zav{\nadsebou{ a \quad b}{1-b+a};-1}=\frac{2^{-a}\Gamma(1+a-b)\sqrt{\pi}}{\Gamma\zav{1-b+\frac{a}{2}}\Gamma\zav{\frac{a+1}{2}}}.
$$
See \cite[15.4.26]{dlmf15.4}.

It might be possible to derive the formula (\ref{2F1(3/4)}) from the known summation formula for $\!\! \ _2 F_1(-1/3)$ in \cite[2.8.53]{Erdelyi}, but the author is unaware at the moment whether the two are related or not. 
\end{example}
\bigskip
In a sense, \textit{there is} a change of variable formula for generic function $y$, but only when parameters $a$, $c$ differ by an integer. 
\begin{proposition} Let $y$ be analytic function near the origin such that $y(0)=0$. Then for all $a\in\mathbb{C}$ and for all $n\in\mathbb{Z}$ it holds:
\begin{equation}\label{gensubsder}
\hypop_{a}^{a+n}(x)=\zav{\frac{x}{y}}^{1-a}\zav{\prod_{j=1}^n y'\hypop_{a+j-1}^{a+j}(y)}\zav{\frac{x}{y}}^{a+n-1}.
\end{equation}
\end{proposition}
\begin{proof} For  $n\in \mathbb{N}$ this is an easy (though tedious) consequence of the formula (\ref{Hdifop}): 
$$
\hypop_{a}^{a+n}(x)=\frac{\zav{a+x\partial_x}_n}{(a)_n},
$$
and the ``change of variable'' formula for derivatives:
$$
x\partial_x=\frac{x}{y} y' y \partial_y.
$$
Once obtain we can invert both sides to get
\begin{equation}\label{gensubsint}
\hypop_{a+n}^{a}(x)=\zav{\frac{x}{y}}^{1-a-n}\zav{\prod_{j=1}^n \hypop_{a+n+1-j}^{a+n-j}(y)\frac{1}{y'}}\zav{\frac{x}{y}}^{a-1}.
\end{equation}
Now rename $a\to a-n$ and we get
$$
\hypop_{a}^{a-n}(x)=\zav{\frac{x}{y}}^{1-a}\zav{\prod_{j=1}^n \hypop_{a+1-j}^{a-j}(y)\frac{1}{y'}}\zav{\frac{x}{y}}^{a-n-1}.
$$
This is exactly the formula (\ref{gensubsder}) for $n=-n$ if we interpret the product as in (\ref{negativeproduct}). Therefore (\ref{gensubsder}) holds for every integer $n$.
\end{proof}
\section{Proof of the main theorem}\label{Sproof}
We are ready to prove Theorem \ref{main}. Let us repeat the statement.

{\bf T\scriptsize HEOREM \normalsize} \ref{main} \emph{ Let 
$$
y=F_m(x):=1-(1-x)^m,\qquad m\in\mathbb{Z}.
$$
Then assuming either
$$
1)\qquad m\in \szav{-2,-1,1,2},\ \forall a,c\in \mathbb{C},\qquad \text{or }\qquad 2)\qquad 
\forall m\in\mathbb{Z}\setminus \szav{0},\ a-c\in\mathbb{Z},  
$$
it holds 
\begin{equation*}
(\ref{Fnsubs})\qquad \hypop_{c}^{a}(x)=
\zav{\frac{mx}{y}}^{1-c}(1-y)^{1+\frac{c-a}{m}}\zav{\prod_{j=1}^{m}(1-y)^{\frac{a-c-1}{m}}\hypop_{c+(j-1)\frac{a-c}{m}}^{c+j\frac{a-c}{m}}(y)}\zav{\frac{mx}{y}}^{a-1}.
\end{equation*}}
\begin{proof}
For $m=1$ we have $F_1(x)=x$ and (\ref{Fnsubs}) trivially holds. 

For $m=-1$ we have $F_{-1}(x)=\frac{x}{x-1}=P(x)$ and (\ref{Fnsubs}) is actually a restatement of the Pfaff property (\ref{Pfaffproperty}).

Cases $m=\pm 2$ follows from Proposition \ref{Qprop} since
$$
F_2(x)=1-(1-x)^2=Q\circ P\circ S_\frac12(x),\qquad F_{-2}(x)=1-\frac{1}{(1-x)^2}=P\circ Q\circ P\circ S_\frac12(x).
$$

What remains is thus to show that the formula (\ref{Fnsubs}) holds for all $m$ when $a-c\in\mathbb{Z}$.
Note that
$$
1-y=(1-x)^{m},\qquad y'=m\zav{1-x}^{m-1}=m\zav{1-y}^{1-\frac{1}{m}}
$$
Thus
\begin{align*}
\hypop_{a}^{a+n}(x)&\stackrel{(\ref{gensubsder})}{=}\zav{\frac{mx}{y}}^{1-a}\zav{\prod_{j=1}^n (1-y)^{1-\frac{1}{m}}\hypop_{a+j-1}^{a+j}(y)}\zav{\frac{mx}{y}}^{a+n-1}.\\
\intertext{Remember, this holds for all integer $n$. Not just positive. We must distinguish two cases depending on the sign of $m$. For $m>0$ we are going to apply the general version of Euler property (\ref{Eulergen}) with $c_j=1-1/m$, $a_j=a+j$ altogether $m-1$ times as in (\ref{Eulergenm}). Note that $c_j=a_j-a_{j-1}-1/m$ so $\alpha=-1/m$. We obtain}
\hypop_{a}^{a+n}(x)&\stackrel{(\ref{Eulergenm})}{=}\zav{\frac{mx}{y}}^{1-a}(1-y)^{\frac{m-1}{m}}\zav{\prod_{j=1}^n (1-y)^{0}\hypop^{a+\frac{j}{m}}_{a+\frac{j-1}{m}}(y)}
\zav{\prod_{k=1}^{m-1}(1-y)^{\frac{n-1}{m}}\hypop_{a+\frac{n}{m}j}^{a+\frac{n}{m}(j+1)}(y)}
\zav{\frac{mx}{y}}^{a+n-1}\\
&\stackrel{(\ref{exchange})}{=}\zav{\frac{mx}{y}}^{1-a}(1-y)^{\frac{m-1}{m}}\hypop_{a}^{a+\frac{n}{m}}(y)
\zav{\prod_{j=1}^{m-1}(1-y)^{\frac{n-1}{m}}\hypop_{a+\frac{n}{m}j}^{a+\frac{n}{m}(j+1)}(y)}
\zav{\frac{mx}{y}}^{a+n-1}\\
&=\zav{\frac{mx}{y}}^{1-a}(1-y)^{\frac{m-n}{m}}
\zav{\prod_{j=1}^{m}(1-y)^{\frac{n-1}{m}}\hypop_{a+\frac{n}{m}(j-1)}^{a+\frac{n}{m}j}(y)}
\zav{\frac{mx}{y}}^{a+n-1}.\\
\intertext{Changing the notation $a\to c$ and $n\to a-c$ we can rewrite the final result as follows:}
\hypop_{a}^{a+n}(x)&=\zav{\frac{mx}{y}}^{1-c}(1-y)^{\frac{m-a+c}{m}}
\zav{\prod_{j=1}^{m}(1-y)^{\frac{a-c-1}{m}}\hypop_{c+\frac{a-c}{m}(j-1)}^{c+\frac{a-c}{m}j}(y)}
\zav{\frac{mx}{y}}^{a-1}.
\end{align*} 
Since the crucial identity (\ref{Eulergenm}) is valid for all integer $n$, this proves (\ref{Fnsubs}) for all $a-c\in\mathbb{Z}$ in the case $m>0$.

For $m<0$ the proof is completely analogous. Starting again with
\begin{align*}
\hypop_{a}^{a+n}(x)&\stackrel{(\ref{gensubsder})}{=}\zav{\frac{mx}{y}}^{1-a}\zav{\prod_{j=1}^n (1-y)^{1-\frac{1}{m}}\hypop_{a+j-1}^{a+j}(y)}\zav{\frac{mx}{y}}^{a+n-1}.\\
\intertext{Now we apply the general version of Euler property (\ref{Eulergen}) with $c_j=1-1/m$, $a_j=a+j$ altogether $1-m$ times as in (\ref{Eulergen-m}). Note that $c_j=a_j-a_{j-1}-1/m$ so $\alpha=-1/m$. We obtain}
\hypop_{a}^{a+n}(x)&\stackrel{(\ref{Eulergen-m})}{=}\zav{\frac{mx}{y}}^{1-a}(1-y)^{\frac{m-1}{m}}\zav{\prod_{j=1}^n (1-y)^{0}\hypop_{a+\frac{j-1}{m}}^{a+\frac{j}{m}}(y)}
\zav{\prod_{j=1}^{1-m}\hypop_{a-\frac{j-2}{m}n}^{a-\frac{j-1}{m}n}(y)(1-y)^{-\frac{n-1}{m}}}
\zav{\frac{mx}{y}}^{a+n-1}\\
&\stackrel{(\ref{exchange})}{=}\zav{\frac{mx}{y}}^{1-a}(1-y)^{\frac{m-1}{m}}\hypop_{a}^{a+\frac{n}{m}}(y)
\zav{\prod_{j=1}^{1-m}\hypop_{a-\frac{j-2}{m}n}^{a-\frac{j-1}{m}n}(y)(1-y)^{-\frac{n-1}{m}}}
\zav{\frac{mx}{y}}^{a+n-1}\\
&\stackrel{(\ref{inverse})}{=}\zav{\frac{mx}{y}}^{1-a}(1-y)^{\frac{m-n}{m}}
\zav{\prod_{j=2}^{1-m}\hypop_{a-\frac{j-2}{m}n}^{a-\frac{j-1}{m}n}(y)(1-y)^{-\frac{n-1}{m}}}
\zav{\frac{mx}{y}}^{a+n-1}\\
&=\zav{\frac{mx}{y}}^{1-a}(1-y)^{\frac{m-n}{m}}
\zav{\prod_{j=1}^{-m}\hypop_{a-\frac{j-1}{m}n}^{a-\frac{j}{m}n}(y)(1-y)^{-\frac{n-1}{m}}}
\zav{\frac{mx}{y}}^{a+n-1}\\
&\stackrel{(\ref{negativeproduct})}{=}\zav{\frac{mx}{y}}^{1-a}(1-y)^{\frac{m-n}{m}}
\zav{\prod_{j=1}^{m}(1-y)^{\frac{n-1}{m}}\hypop^{a+\frac{j}{m}n}_{a+\frac{j-1}{m}n}(y)(1-y)^{-\frac{n-1}{m}}}
\zav{\frac{mx}{y}}^{a+n-1}.\\
\end{align*}
This is exactly the same result as before but for negative $m$. This therefore proves our result (\ref{Fnsubs}) for all integer $m$ and for parameters $a$, $c$ such that $a-c\in\mathbb{Z}$. 

\end{proof}

\begin{example}
We can combine the function $F_n$ with $S_\alpha$ and $M_\alpha$ to obtain additional interesting formulas. For instance, we can recover the following ``cubic'' transform: 
Let
$$
y(x):=F_3\circ S_{\frac32}\circ P\circ S_2(x)=1-\zav{\frac{1-x}{1+2x}}^3=\frac{9x(1-x^3)}{(1-x)(1+2x)^3}.
$$
Then
\begin{eqnarray}\label{conj1eq}
(1+2x)^{a+3c-3}\zav{1-x^3}^{1-c}\hypop_{c}^{a}(x)(1+2x)^{3-3a-c}\zav{1-x^3}^{a-1}=\\
(1-y)^{1-\frac{c}{3}}\hypop_{c}^{\frac{2+a+2c}{3}}(y)(1-y)^{-\frac{c-a}{3}}\hypop_{\frac{2+2c+a}{3}}^{\frac{1+2a+c}{3}}(y)(1-y)^{-\frac{c-a}{3}}\hypop_{\frac{1+c+2a}{3}}^{a}(y)(1-y)^{\frac{a}{3}-1}.\nonumber
\end{eqnarray}
Right now, this formula holds only for $a-c\in\mathbb{Z}$. But granted it is true for all $a-c$, it should be in principle possible to obtain the following well known identity:
$$
\!\!\ _2 F_1\zav{\nadsebou{\frac13\quad \frac23}{1};1-\zav{\frac{1-x}{1+2x}}^3}=(1+2x)
\!\!\ _2 F_1\zav{\nadsebou{\frac13\quad \frac23}{1};x^3},
$$
i.e. the Ramanujan's cubic transform \cite[15.8.33]{dlmf15.8}. But the author is currently unable to do so.
\end{example}
\begin{example}
Also
let
$$
y(x):=F_3\circ S_2\circ P(x)=1-\zav{\frac{1-x}{1+x}}^3=\frac{2x(3+x^2)}{(1+x)^3}.
$$
Then
\begin{eqnarray}\label{conj2eq}
(1-x)^{1-c}\zav{1+x}^{3c+a-3}\zav{1+\frac{x^2}{3}}^{1-c}\hypop_{c}^{a}(x)(1-x)^{a-1}(1+x)^{3-3a+c}\zav{1+\frac{x^2}{3}}^{a-1}\\
=(1-y)^{1-\frac{c}{3}}\hypop_{c}^{\frac{2+a+2c}{3}}(y)(1-y)^{-\frac{c-a}{3}}\hypop_{\frac{2+2c+a}{3}}^{\frac{1+2a+c}{3}}(y)(1-y)^{-\frac{c-a}{3}}\hypop_{\frac{1+c+2a}{3}}^{a}(y)(1-y)^{\frac{a}{3}-1}.\nonumber
\end{eqnarray}
We can verify this formula on a specific functions. 
Applying (\ref{conj2eq}) with $c=-2a$ on the function
$$
(1-y)^{1-\frac{a}{3}},
$$
and then replacing $a\to -a$ we obtain
\begin{equation}\label{2F1cubicerd}
\!\! \ _2 F_1\zav{\nadsebou{a\quad a+\frac13}{2a};\frac{2x(3+x^2)}{(1+x)^3}}=
(1-x)^{1-a}(1+x)^{-3a-3}\zav{1+\frac{x^2}{3}}^{1-2a}\hypop_{2a}^{-a}(x)(1-x)^2\zav{1+\frac{x^2}{3}}^{-a-1}.
\end{equation}
Expanding the term $(1-x)^2$ and performing hypergeometrization we do obtain a cubic transform of $\!\!\ _2 F_1$ which can be found in \cite[(2.11.39)]{Erdelyi}. This is therefore a supporting evidence for validity of Conjecture~\ref{Conjecture}.
\end{example}
\section{Acknowledgment}

The author was supported by the GA\v CR grant no. 21-27941S and RVO funding 47813059.

\end{document}